\documentclass{article}
\usepackage{latexsym,amsfonts,amsthm,amsmath,amscd,amssymb}
\usepackage[dvips]{graphicx}

\newtheorem{lemma}{Lemma}[section]
\newtheorem{thm}[lemma]{Theorem}
\newtheorem{rem}[lemma]{Remark}
\newtheorem{prop}[lemma]{Proposition}

\newtheorem{oss}[lemma]{Observation}
\newtheorem{example}[lemma]{Example}
\newtheorem{defn}[lemma]{Definition}

\newcommand\matR{{\mathbb{R}}}
\newcommand\matN{{\mathbb{N}}}

\renewcommand{\hbar}{{\overline{h}}}

\newfont{\Got}{eufm10 scaled 1200}

\newcommand\calD{{\mathcal D}}

\begin{document}

\title{Diffeological vector pseudo-bundles}

\author{Ekaterina~{\textsc Pervova}}

\maketitle

\begin{abstract}
\noindent We consider a diffeological counterpart of the notion of a
vector bundle (we call this counterpart a pseudo-bundle, although in
the other works it is called differently; among the existing terms
there are a ``regular vector bundle'' of Vincent and ``diffeological
vector space over X'' of Christensen-Wu). The main difference of the
diffeological version is that (for reasons stemming from the
independent appearance of this concept elsewhere), diffeological
vector pseudo-bundles may easily not be locally trivial (and we
provide various examples of such, including those where the
underlying topological bundle is even trivial). Since this precludes
using local trivializations to carry out many typical constructions
done with vector bundles (but not the existence of constructions
themselves), we consider the notion of diffeological gluing of
pseudo-bundles, which, albeit with various limitations that we
indicate, provides when applicable a substitute for said local
trivializations. We quickly discuss the interactions between the
operation of gluing and typical operations on vector bundles (direct
sum, tensor product, taking duals) and then consider the notion of a
pseudo-metric on a diffeological vector pseudo-bundle.

\end{abstract}

\section*{Introduction}

Diffeology as a subject, introduced by Souriau in the 80's
\cite{So1,So2}, belongs among various attempts made over the years
to extend the usual setting of differential calculus and/or
differential geometry. Many of these attempts appeared in the realm
of mathematical physics, such as smooth structures \`a la Sikorski
or \`a la Fr\"olicher, and were motivated by the fact that many
objects that naturally appear in, for example, noncommutative
geometry, such as irrational tori, orbifolds, spaces of connections
on principal bundles in Yang-Mills theory, and so on, are not smooth
manifolds and cannot be easily treated by similar methods. A rather
comprehensive summary of various attempts of extending, in a
consistent way, the category of smooth manifolds can be found in
\cite{St}.

Among such attempts, the diffeology has (at least) the virtue of
being an essentially very simple construction, possibly appealing to
those not very experienced with heavy analytic matters and more
pointed towards geometric setting. As an instance, one finds, at
first glance, that many concepts extend to this category almost
\emph{verbatim} (in any case, in an obvious way). But on the other
hand, once again with very little effort, one notices, be that with
simple examples or a single construction coming from elsewhere, that
the trivial extension is unsatisfactory.

The concept of the vector bundle is an excellent instance of this. A
trivial extension of this concept just requires substituting each
mentioning of a smooth map with ``diffeologically smooth'', and
maybe choosing a diffeology on the fibre (but then, in the
finite-dimensional case, which we are limiting ourselves to, there
is a standard choice). Yet, it is easily seen that this is not
sufficient; one way to to explain why is to point to the internal
tangent bundles of Christensen-Wu \cite{CWtangent}. These,
frequently enough, turn out not to be vector bundles at all, for the
simple reason that they do not always have the same fibre (for
reasons related to the underlying topology of the base space), yet,
they are more than legitimate candidates to be tangent bundles.

Thus, the aim of this note is to take a closer look at the objects
of this type, starting from attempting to define with precision what
``of this type'' means (apart from the very general definition given
in \cite{CWtangent}). In particular, we explore the path of
constructing these ``pseudo-bundles'' (as we call them) by a kind of
successive gluing; and also attempt to give local descriptions. None
of these might lead to a completely satisfactory final answer; but
it is worth a try.

\paragraph{The specific issues} We start from a specific example,
due to Christensen-Wu (\cite{CWtangent}, Example 4.3); it has a
merit that it points out right away what, very informally, can be
described as the first main contribution of diffeology to the
mathematical landscape: the possibility to treat topological spaces
which are in no way smooth manifolds (not even having a manifold's
topology), as if they were such, and to do so in a uniform
manner.\footnote{A disclaimer is necessary at this point. First of
all, the above is just an opinion by the author. Second, of course
manifolds with corners (for instance) already fall under the above
description, and are already treated by different methods. The
distinction being made is, maybe, that diffeology applies to a much
wider range of objects, and has a different point of view on them.}
Namely, the example cited is the space $X$ that consists of the
coordinate axes in $\matR^2$; it has a kind of smooth structure, an
atlas, if you wish, where the charts are restrictions of all usual
$\matR^2$-valued maps. This sort of structure is an example of a
diffeology on a space (in the sense of a diffeological structure).
For such, a so-called internal tangent bundle (\cite{CWtangent}) is
defined, and for this specific $X$ it reveals itself to be something
very similar to a typical vector bundle (in fact, it is one
everywhere outside the origin, with fibre $\matR$), but it is not
one because the internal tangent space at the origin is $\matR^2$.

This example is significant for more than one reason. Just to
mention two specific ones, we observe, first, that the tangent space
at the origin being $2$-dimensional, while being $1$-dimensional
elsewhere, seems to be a necessity indeed, since it reflects the
usual topological structure of the space; anything else would be
counterintuitive.\footnote{This reason is pretty much a statement of
the obvious.} Secondly, the internal tangent bundles possessing a
certain multiplicativity property, just starting from this specific
space and taking its direct product, not even with itself, but with
any $\matR^n$, we shall find similar examples, in any dimension, and
with a more complicated structure. These observations bring us to
the next paragraph.

\paragraph{The aims} What we wish to do in this paper, is to take
an abstract look at the vector ``bundles'' of the above-described
type (the precise definition, that of \emph{diffeological vector
space over $X$}, is given later in the paper; it is however almost
that of a typical vector bundle, but does not include the
requirement of being locally trivial). Informally speaking, we would
like to give a more concrete characterization of such
pseudo-bundles; to this end, we consider various types of
topological operations on them,\footnote{The usual vector bundle
operations, such as taking the direct sum, the tensor product, and
the dual bundle were already described in \cite{vincent}; we briefly
recall them.} such as gluing them together. This point of view,
which is not entirely general, has to do with possible presentations
of the base space as a ``simplicial complex''; by this, we do not
mean (not necessarily) a simplicial structure in the strict sense,
but rather a decomposition of the base space into copies of
Euclidean spaces (of different dimensions).\footnote{We note right
away that this does not yield a satisfactory construction right
away; indeed, the diffeology that arises immediately from such
gluing, is too much related to the specific decomposition used, and
this is one issue that needs to be dealt with.} The hope is to
arrive to some kind of local description, which would be a
diffeological counterpart of local trivializations for usual vector
bundles.

\paragraph{Acknowledgments} Changing fields and starting anew
elsewhere is never easy; and life, in all its complexity, may or may
not collaborate with such an endeavour (and frequently does not). At
such a moment, especially at the start but even further, it is of
the greatest value whoever, at whatever circumstance, shows a
support towards this struggle of yours, be that an encouragement
\emph{a posteriori}, or a simple comment that changing things is
good. I take advantage of this new piece of work to heartily thank
two people who, to me, such support did show, Prof. Riccardo Zucchi
and Dr. Elisabetta Chericoni.

\section{Diffeology and diffeological vector spaces}%rewrite this

To make the paper self-contained, we collect here the definitions of
all the main objects that appear in the sequel.

\subsection{Diffeologies on sets}

The basic notion is that of a diffeological space and a diffeology
on it, together with some particular types of maps between
diffeological spaces; we follow \cite{iglesiasBook}.

\paragraph{The concept} The definition of a diffeological space and
its diffeological structure (or, briefly, its diffeology) is as
follows.

\begin{defn} \emph{(\cite{So2})} A \textbf{diffeological space} is a pair
$(X,\calD_X)$ where $X$ is a set and $\calD_X$ is a specified
collection of maps $U\to X$ (called \textbf{plots}) for each open
set $U$ in $\matR^n$ and for each $n\in\matN$, such that for all
open subsets $U\subseteq\matR^n$ and $V\subseteq\matR^m$ the
following three conditions are satisfied:
\begin{enumerate}
  \item (The covering condition) Every constant map $U\to X$ is a
  plot;
  \item (The smooth compatibility condition) If $U\to X$ is a plot
  and $V\to U$ is a smooth map (in the usual sense) then the
  composition $V\to U\to X$ is also a plot;
  \item (The sheaf condition) If $U=\cup_iU_i$ is an open cover and
  $U\to X$ is a set map such that each restriction $U_i\to X$ is a
  plot then the entire map $U\to X$ is a plot as well.
\end{enumerate}
\end{defn}

When the context permits, instead of $(X,\calD_X)$ we simply write
simply $X$.

\begin{defn} \emph{(\cite{So2})} Let $X$ and $Y$ be two
diffeological spaces, and let $f:X\to Y$ be a set map. We say that
$f$ is \textbf{smooth} if for every plot $p:U\to X$ of $X$ the
composition $f\circ p$ is a plot of $Y$.
\end{defn}

The diffeological counterpart of an isomorphism between
diffeological spaces is (expectedly) called a
\textbf{diffeomorphism}; there is a typical notation
$C^{\infty}(X,Y)$ which denotes the set of all smooth maps from $X$
to $Y$. An obvious example of a diffeological space is any smooth
manifold, whose diffeology consists of all usual smooth maps; then a
diffeomorphism in the diffeological sense is the same thing as a
diffeomorphism in the usual sense.

\paragraph{The D-topology} There is a canonical topology underlying
every diffeological structure on a given set, the so-called
D-topology.\footnote{A frequent restriction for the choice of a
diffeology on a given topological space is that the corresponding
D-topology coincide with the given one.} It is defined
(\cite{iglesiasBook}) as the final topology on a diffeological space
$X$ induced by its plots, where each domain is equipped with the
standard topology. To be more explicit, if $(X,\calD_X)$ is a
diffeological space then a subset $A$ of $X$ is open in the
D-topology of $X$ if and only if $p^{-1}(A)$ is open for each
$p\in\calD_X$; such subsets are called \textbf{D-open}. In the case
of a smooth manifold with the standard diffeology, the D-topology is
the same as the usual topology on the manifold, and this is
frequently the case also for non-standard diffeologies. This is due
to the fact that, as established in \cite{CSW_Dtopology}, Theorem
3.7, the D-topology is completely determined smooth curves, in the
sense that a subset $A$ of $X$ is D-open if and only if $p^{-1}(A)$
is open for every $p\in C^{\infty}(\matR,X)$.

\paragraph{Comparing diffeologies} In a way somewhat similar as it
occurs for the set of all possible topologies on a given set $X$,
the set of all possible diffeologies on $X$ is partially ordered by
inclusion. Specifically, a diffeology $\calD$ on $X$ is said to be
\textbf{finer} than another diffeology $\calD'$ if
$\calD\subset\calD'$, while $\calD'$ is said to be \textbf{coarser}
than $\calD$. Among all diffeologies, there is the finest one (the
natural \textbf{discrete diffeology}, which consists of all locally
constant maps $U\to X$) and the coarsest one (which consists of
\emph{all} possible maps $U\to X$, for all $U\subseteq\matR^n$ and
for all $n\in\matN$ and is called the \textbf{coarse diffeology}).
In other words, the set of all diffeologies forms a complete
lattice.

\paragraph{Constructing diffeologies by bounds} The above-mentioned
structure of a lattice on the set of all diffeologies on a given $X$
is frequently employed when constructing (or defining) a desired
diffeology, for instance, one that contains a given plot, or one
that includes only plots that, as maps, enjoy a certain specified
property. For such restricted sets of diffeologies, it is frequently
possible to claim the existence of the smallest/finest (or the
largest/coarsest) diffeology among them; some of the definitions
that follow (for example, those of the \emph{sum diffeology} and of
the \emph{product diffeology}) are instances of this.

\paragraph{The generated diffeology} This is a particularly
important, from the practical point of view, at least, instance of
the above-mentioned use of bounds to construct a diffeology. We
stress one observation that trivially follows the concept of a
generated diffeology: for any set $X$ and any map $p:U\to X$ defined
on a domain $U\subset\matR^k$ (and for any $k$), there is a
diffeology on $X$ for which $p$ is a plot; that is, \emph{any} map
can be seen as a smooth map in the diffeological
setting.\footnote{Such breadth might have its own disadvantages, of
course.}

Let us now state the precise definition. Given a set $X$ and a set
of maps $\mathcal{A}=\{U\to X\}$, all defined on some domains of
some $\matR^m$'s, there exists the finest diffeology on $X$ that
contains $\mathcal{A}$. This diffeology is called the
\textbf{diffeology generated by $\mathcal{A}$}.

\paragraph{Pushforwards and pullbacks of diffeologies} Let $X$ be a
diffeological space, $X'$ an arbitrary set, and let $f:X\to X'$ be
any map. Then there exists a finest diffeology on $X'$ that makes
the map $f$ smooth; this diffeology is called the
\textbf{pushforward of the diffeology of $X$ by the map $f$} and is
denoted by $f_*(\calD)$, where $\calD$ stands for the diffeology of
$X$. Furthermore, if we have a reverse situation, \emph{i.e.}, if
$X$ is just a set and $X'$ is a diffeological space with diffeology
$\calD'$, then there is the \textbf{pullback} of the diffeology
$\calD'$ by a given map $f:X\to X'$: it is the coarsest diffeology
on $X$ such that $f$ is smooth. The pullback diffeology is denoted
by $f^*(\calD')$.

\paragraph{The quotient diffeology} A quotient of a diffeological
space is always a diffeological space\footnote{Unlike smooth
manifolds, whose quotients frequently are not manifolds at all.} for
a canonical choice of a diffeology on the quotient. Namely, let $X$
be a diffeological space, let $\cong$ be an equivalence relation on
$X$, and let $\pi:X\to Y:=X/\cong$ be the quotient map; the
\textbf{quotient diffeology} on $Y$ is the pushforward of the
diffeology of $X$ by the natural projection (which is automatically
smooth). It can also be described explicitly as follows: $p:U\to Y$
is a plot for the quotient diffeology if and only for each point in
$U$ there exist a neighbourhood $V\subset U$ and a plot
$\tilde{p}:V\to X$ such that $p|_{V}=\pi\circ\tilde{p}$.

\paragraph{The subset diffeology, inductions and subductions}
Let $X$ be a diffeological space, and let $Y\subseteq X$ be its
subset. The \textbf{subset diffeology} on $Y$ is the coarsest
diffeology on $Y$ making the inclusion map $Y\hookrightarrow X$
smooth. It consists of all maps $U\to Y$ such that $U\to
Y\hookrightarrow X$ is a plot of $X$. This notion is frequently used
in practice and makes part of further definitions, such as the
following ones: for two diffeological spaces $X,X'$ a smooth map
$f:X'\to X$ is called an \textbf{induction} if it induces a
diffeomorphism $X\to\mbox{Im}(f)$, where $\mbox{Im}(f)$ has the
subset diffeology of $X$; a map $f:X\to X'$ is said to be a
\textbf{subduction} if it is surjective and the diffeology $\calD'$
of $X'$ is the pushforward of the diffeology $\calD$ of $X$.

\paragraph{Disjoint sums and products} Let $\{X_i\}_{i\in I}$ be a
collection of diffeological spaces, where $I$ is a set of indices.
The \textbf{sum} of $\{X_i\}_{i\in I}$ is defined as
$$X=\coprod_{i\in I}X_i=\{(i,x)\,|\,i\in I\mbox{ and }x\in X_i\}.$$
The \textbf{sum diffeology} on $X$ is the \emph{finest} diffeology
such that each natural injection $X_i\to\coprod_{i\in I}X_i$ is
smooth; it consists of plots that locally are plots of one of the
components of the sum. The \textbf{product diffeology} $\calD$ on
the product $\prod_{i\in I}X_i$ is the \emph{coarsest} diffeology
such that for each index $i\in I$ the natural projection
$\pi_i:\prod_{i\in I}X_i\to X_i$ is smooth; locally, it consists of
tuples of plots of all the components of the product.

\paragraph{Functional diffeology} Let $X$, $Y$ be two diffeological
spaces, and let $C^{\infty}(X,Y)$ be the set of smooth maps from $X$
to $Y$. Let \textsc{ev} be the \emph{evaluation map}, defined by
$$\mbox{\textsc{ev}}:C^{\infty}(X,Y)\times X\to Y\mbox{ and }\mbox{\textsc{ev}}(f,x)=f(x). $$
The \textbf{functional diffeology} is the coarsest diffeology on
$C^{\infty}(X,Y)$ such that this evaluation map is smooth.

\subsection{Diffeological vector spaces}

The concept in itself is quite straightforward: it is a set $X$ that
is both a diffeological space and a vector space such that the
operations are smooth (with respect to the diffeology).

\paragraph{The concept and some basic constructions} Let $V$ be a
vector space over $\matR$. A \textbf{vector space diffeology} on $V$
is any diffeology of $V$ such that the addition and the scalar
multiplication are smooth, that is,
$$[(u,v)\mapsto u+v]\in C^{\infty}(V\times V,V)\mbox{ and }[(\lambda,v)\mapsto\lambda v]\in C^{\infty}(\matR\times V,V),$$
where $V\times V$ and $\matR\times V$ are equipped with the product
diffeology.\footnote{Note that $\matR$ has standard diffeology here,
a fact that has significant implications for what a vector space
diffeology could be; the most obvious of those is that the discrete
diffeology is never a vector space diffeology (except for the zero
space), see the explanation below.} A \textbf{diffeological vector
space} over $\matR$ is any vector space $V$ over $\matR$ equipped
with a vector space diffeology.

The following observation could be useful to clarify the concept.
Since the constant maps are plots for any diffeology and the scalar
multiplication is smooth with respect to the standard diffeology of
$\matR$, any vector space diffeology on a given $V$ includes maps of
form $f(x)v$ for any fixed $v\in V$ and for any smooth map
$f:\matR\to\matR$. Next, since the addition is smooth, any vector
space diffeology includes all finite sums of such maps. This
immediately implies that any vector space diffeology on $\matR^n$
includes all usual smooth maps (since they write as
$\sum_{i=1}^nf_i(x)e_i$).\footnote{This is not the case for a
non-vector space diffeology of $\matR^n$; the simplest example is
the already-mentioned discrete diffeology, for which the scalar
multiplication is not smooth. A more intricate example is that of
the so-called \emph{wire diffeology}, one generated by the set
$C^{\infty}(\matR,\matR^n)$. For this diffeology, the scalar
multiplication is smooth, but the addition is not.}

All the usual constructions of linear algebra, such as spaces of
(smooth) linear maps, products, subspaces, and quotients, are
present in the category of diffeological vector spaces. Obviously,
given two diffeological vector spaces $V$ and $W$, one speaks of the
space of \textbf{smooth linear maps} between them; this space is
denoted by $L^{\infty}(V,W)$ and is defined simply as:
$$L^{\infty}(V,W)=L(V,W)\cap C^{\infty}(V,W);$$
this is an $\matR$-linear subspace of $L(V,W)$ and is \emph{a
priori} smaller than the whole space $L(V,W)$.\footnote{It is very
easy to give examples where it is strictly smaller; consider, for
instance, $\matR^n$ with the vector space diffeology generated by a
plot of form $\matR\ni x\mapsto f(x)e_n$, where $f(x)$ is any
non-differentiable function. Then the usual linear dual of $e_n$ is
linear but not smooth.} A \textbf{subspace} of a diffeological
vector space $V$ is a vector subspace of $V$ endowed with the subset
diffeology. It is easy to see (\cite{iglesiasBook}, Section 3.5)
that if $V$ is a diffeological vector space and $W\leqslant V$ is a
subspace of it then the quotient $V/W$ is a diffeological vector
space with respect to the quotient diffeology.

\paragraph{The direct sum/product of diffeological vector spaces} Let
$\{V_i\}_{i\in I}$ be a family of diffeological vector spaces.
Consider the usual direct sum $V=\oplus_{i\in I}V_i$ of this family;
then $V$, equipped with the product diffeology, is a diffeological
vector space.

\paragraph{Euclidean structure on diffeological vector spaces} The
notion of a \textbf{Euclidean diffeological vector space} does not
differ much from the usual notion of the Euclidean vector space. A
diffeological space $V$ is Euclidean if it is endowed with a scalar
product that is smooth with respect to the diffeology of $V$ and the
standard diffeology of $\matR$; that is, if there is a fixed map
$\langle , \rangle:V\times V\to\matR$ that has the usual properties
of bilinearity, symmetricity, and definite-positiveness and that is
smooth with respect to the diffeological product structure on
$V\times V$ and the standard diffeology on $\matR$. We will speak in
more detail of this later on, but it is worthwhile pointing out
right away that, although many diffeological vector spaces admit
plenty of smooth bilinear symmetric forms, a
\emph{finite-dimensional diffeological vector space admits a smooth
scalar product if and only if it is diffeomorphic to some $\matR^n$
with the standard diffeology} (see \cite{iglesiasBook}, Ex. 70 on p.
74 and its solution). For other finite-dimensional diffeological
vector spaces a kind of ``minimally degenerate'' smooth symmetric
bilinear form can be considered (see
\cite{pseudometric}).\footnote{On the other hand, obtaining a
finite-dimensional diffeological vector space where the only smooth
bilinear form is the zero form is also easy: take $\matR^n$ and $n$
non-differentiable functions $f_1(x)$, $\ldots$, $f_n(x)$. Then
endowing $\matR^n$ with the vector space diffeology generated by the
$n$ plots $\matR\ni x\mapsto f_i(x)e_i$ yields a diffeological
vector space where the only smooth (multi)linear map is the zero
map.}

\paragraph{Fine diffeology on vector spaces} The \textbf{fine
diffeology} on a vector space $\matR$ is the \emph{finest} vector
space diffeology on it; endowed with such, $V$ is called a
\emph{fine vector space}. Note that \emph{any} linear map between
two fine vector spaces is smooth (\cite{iglesiasBook}, 3.9). An
example of a fine vector space is $\matR^n$ with the standard
diffeology, \emph{i.e.}, one that consists of all the usual smooth
maps with values in $\matR^n$.\footnote{It is easy to see that this
set is indeed a (vector space) diffeology. Furthermore, it is the
finest one, since, as we have already observed above, it is
contained in any other vector space diffeology.}

\paragraph{The dual of a diffeological vector space} The definition
of the diffeological dual was first given in \cite{vincent} and then
in \cite{wu}; this concept is a very natural (and obvious) one:

\begin{defn}
Let $V$ be a diffeological vector space. The \textbf{diffeological
dual} of $V$, denoted by $V^*$, is the set $L^{\infty}(V,\matR)$ of
all smooth linear maps $V\to\matR$.
\end{defn}

The resulting space is a diffeological vector space for the
functional diffeology; in general it is not isomorphic to $V$.
Indeed, as shown in \cite{pseudometric}, the functional diffeology
of the diffeological dual of a finite-dimensional diffeological
vector space is always the standard one (in particular, the dual of
any space is a fine space). Hence, in the finite-dimensional case
the equality $L^{\infty}(V,\matR)=L(V,\matR)$ holds, and so $V^*$
and $V$ are isomorphic, \emph{if and only if} $V$ is a standard
space.\footnote{Note also that, as shown in \cite{multilinear},
Proposition 4.4, if $V^*$ and $V$ are isomorphic then they are also
diffeomorphic.} The matters become less straightforward in the
infinite-dimensional case, which in this paper we do not consider.

\paragraph{The tensor product} The definition of the diffeological
tensor product was given first in \cite{vincent} and then in
\cite{wu} (see Section 3); we recall the latter version. Let $V_1$,
..., $V_n$ be diffeological vector spaces, let
$T:V_1\times\ldots\times V_n\to V_1\otimes\ldots\otimes V_n$ be the
universal map onto their tensor product as vector spaces, and let
$Z\leqslant V_1\times\ldots\times V_n$ be the kernel of $T$. The
tensor product $\calD_{\otimes}$ on $V_1\otimes\ldots\otimes V_n$ is
the quotient diffeology on $V_1\otimes\ldots\otimes
V_n=(V_1\times\ldots\times V_n)/Z$ coming from the product
diffeology on $V_1\times\ldots\times V_n$. The diffeological tensor
product thus defined possesses the usual universal property (see
\cite{vincent}, Theorem 2.3.5):
$$L^{\infty}(V_1\otimes\ldots\otimes V_n,W)\cong\mbox{Mult}^{\infty}(V_1\times\ldots\times V_n,W),$$
where $\mbox{Mult}^{\infty}(V_1\times\ldots\times V_n,W)$ is the
space of all smooth (with respect to the product diffeology)
multilinear maps $V_1\times\ldots\times V_n\to W$.

\subsection{Diffeological bundles and pseudo-bundles}

A smooth surjective map $\pi:T\to B$ is a \textbf{fibration} if
there exists a diffeological space $F$ such that the pullback of
$\pi$ by any plot $p$ of $B$ is locally trivial, with fibre $F$. The
latter condition has the obvious meaning, namely that there is a
cover of $B$ by a family of D-open sets $\{U_i\}_{i\in I}$ such that
the restriction of $\pi$ over each $U_i$ is trivial with fibre $F$.
For the sake of completeness we mention that there is also another
definition of a diffeological fibre bundle (\cite{iglesiasBook},
8.8), which involves the notion of a \emph{diffeological groupoid};
we use the definition given above since it is more practical.

One point that should be stressed right away (even if it is quite
obvious) is that the condition of local triviality that
$\pi^{-1}(U_i)$ is diffeomorphic to $U_i\times F$ as diffeological
spaces. The reason why we stress this is that it might easily
happen, and it does even for very simple examples (which we present
below), that the two spaces are homeomorphic, even diffeomorphic, in
the usual sense, but they are not in the diffeological sense. It
might even happen that a bundle trivial from the usual point of view
is not even locally trivial with respect to diffeologies involved
(with, for instance, but not only, an isolated fibre carrying a
different diffeology).

\paragraph{Principal diffeological fibre bundles} Let $X$ be a
diffeological space, and let $g\mapsto g_X$ be a smooth action of a
diffeological group $G$ on $X$, that is, a smooth homomorphism from
$G$ to $\mbox{Diff}(X)$. Let $F$ be the \emph{action map}:
$$F:X\times G\to X\times X\mbox{ with }F(x,g)=(x,g_X(x)).$$
Then the following is true (see the Proposition in Section 8.11 of
\cite{iglesiasBook}): if $F$ is an induction then the projection
$\pi$ from $X$ to its quotient $X/G$ is a diffeological fibration,
with the group $G$ as fibre. In this case we say that the action of
$G$ on $X$ is \textbf{principal}. Now, if a surjection $\pi:X\to Q$
is equivalent to $\mbox{class}:X\to G/H$, that is, if there exists a
diffeomorphism $\varphi:G/H\to Q$ such that
$\pi=\varphi\circ\mbox{class}$, we shall say that $\pi$ is a
\textbf{principal fibration}, or a \textbf{principal fibre bundle},
with structure group $G$.

\paragraph{Diffeological vector space over a given $X$} What would
be a \emph{verbatim} extension of the concept of a vector bundle
into the diffeological setting, which in particular would be a
partial case of the definition from the previous paragraph, does not
turn out to be sufficient (some reasons for this have been outlined
in the introduction). This prospective obvious extension is
therefore replaced by the following concept.

\begin{defn} \emph{(\cite{CWtangent}, Definition 4.5)}
Let $X$ be a diffeological space. A \textbf{diffeological vector
space over $X$} is a pair $(V,\pi)$ consisting of a diffeological
space $V$ and a smooth map $\pi:V\to X$ such that each of the fibres
$\pi^{-1}(x)$ is endowed with a vector space structure for which the
following properties hold: 1) the addition map $V\times_X V\to V$ is
smooth with respect to the diffeology of $V$ and the subset
diffeology on $V\times_X V$ coming from the product diffeology on
$V\times V$; 2) the scalar multiplication map $\matR\times V\to V$
is smooth for the product diffeology on $\matR\times V$; 3) the zero
section $X\to V$ is smooth.
\end{defn}

Note that if $X$ is a point, $V$ is just a diffeological vector
space. Furthermore, if $V$ is a diffeological vector space over $X$
then each fibre $p^{-1}(x)$ endowed with the subset diffeology is
automatically a diffeological vector space.

\paragraph{Examples} To illustrate the concept just introduced, we
provide two examples. The first one deals with the case of the most
standard fibration, that of $\matR^n$ over $\matR^k$ (with $k<n$)
via the projection onto a subset of the coordinates of the former;
the second one is more intricate and is specific to the
diffeological version.

\begin{example}\label{coarse:fibres:ex}
Let $V=\matR^n$, and let $\{e_1,\ldots,e_n\}$ be its canonical
basis. Denote by $X$ the subspace generated by the first $k$ vectors
of this basis, and let $\pi$ be the projection of $V$ onto $X$
(\emph{i.e.}, onto the first $k$ coordinates). Obviously, the
pre-image $\pi^{-1}(x)$ of any point $x\in X$ has a natural vector
space structure, which is obtained by representing $\matR^n$ as the
direct product $\matR^k\times\matR^{n-k}$; the fibre $\pi^{-1}(x)$
has then the form $\{x\}\times\matR^{n-k}$, and the vector space
structure is inherited from the second factor.

The space $X$ being canonically identified with $\matR^k$, we endow
it with the standard diffeology. Consider the pullback to $V$ of
this diffeology by the map $\pi$. Writing a plot $p:U\to V$ of this
diffeology as $p(u)=(p_1(u),\ldots,p_n(u))$, and recalling that,
one, $\pi\circ p$ is a plot of $X$ (so it is a usual smooth map)
and, second, the pullback diffeology is the coarsest one with the
latter property, we conclude that $p_1,\ldots,p_k$ must be usual
smooth $\matR$-valued maps, while $p_{k+1},\ldots,p_n$ can be any
maps. In particular, every fibre has coarse
diffeology.\footnote{Note that this example is somewhat artificial;
the pullback diffeology, as we have just reminded, is the largest
possible for which we have a smooth fibration, and presumably the
fibrations that arise from, say, applications would carry a smaller
(more sensible) diffeology. We put this example to illustrate the
extremes of the definition as stated.}
\end{example}

The second example we provide, stems from the fact that the
condition of the local triviality is absent from the definition of a
diffeological vector space over a given $X$; it shows that, in
addition to a large diffeology on the fibres, the definition as
given allows for topologically complicated (in the sense of the
usual topology) total spaces.

\begin{example}\label{easy:pseudobundle:ex}
Let us describe an example of a diffeological vector space over
$X=\{(x,y)\in\matR^2\,|\,xy=0\}$, the union of coordinate axes (the
space that appears in the Christensen-Wu example). Let us construct
$V$ as follows. Take three copies of $\matR^2$ (to distinguish among
them, we denote them by $V_1$, $V_2$, and $V_0$). Consider
$f_1:V_1\supset\{(0,y)\}\to V_0$ acting by $f_1(0,y)=(0,y)$, and,
analogously, $f_2:V_2\supset\{(x,0)\}\to V_0$ acting by
$f_2(x,0)=(x,0)$. Set $V$ to be the result of gluing of $V_1$ and
$V_2$ to $V_0$ via the maps $f_1$ and $f_2$
respectively.\footnote{What we mean here is the usual topological
gluing: given two topological spaces $X$ and $Y$ and a continuous
map $f:X\supset Z\to Y$, the result of gluing $X$ to $Y$ along $f$
is the space $X\cup_f Y:=(X\sqcup Y)/_{z=f(z)}$.}

The resulting space $V$ has a natural projection $\pi$ onto
$X=\{(x,0)\}\cup\{(0,y)\}$, defined by sending every point of $V_1$
to its projection on the $x$-axis: $(x,y)\mapsto(x,0)$, every point
of $V_2$ to its projection on the $y$-axis; $(x,y)\mapsto(0,y)$, and
the whole of $V_0$ to the origin: $(x,y)\mapsto(0,0)$. Note that
$\pi$ is well-defined with respect to the gluing. Observe also that
the pre-image of any point of $X$ has an obvious vector space
structure (obtained in the same way as in the previous example).

Now, the space $X$ is endowed with the subset diffeology $\calD_X$
of $\matR^2$. Let us now consider the pullback $\calD_V$ of this
diffeology by the map $\pi$.\footnote{It is not immediately clear
whether this pullback would necessarily make $V$ into a
diffeological vector space over $X$; however, whichever the case,
there is a standard way, due to Christensen-Wu \cite{CWtangent}, to
make it into such, described later in the paper.} We now show that
the subset diffeology on the fibres is the coarse diffeology.

Note first of all that, since $\pi\circ p$ is a plot of $\calD_X$,
the image of $p$ is contained in either $V_1\cup V_0\subset V$ or
$V_2\cup V_0\subset V$, but not in both. Furthermore, if we assume
that its image is wholly contained in $V_0$ then it can be
\emph{any} map with values in $\matR^2$. Thus, the fibre at the
origin has the coarse diffeology.

Consider now $(x,0)\in X$ with $x\neq 0$; let $p:U\to\matR^2=V_1$ be
a plot for the pullback diffeology for which we assume that
$\mbox{Im}(p)$ intersects $\pi^{-1}(x,0)$ and is wholly contained in
$V_1\subset V$. Write $p(u)=(p_1(u),p_2(u))$. We have $\pi\circ
p(u)=p_1(u)$, which by definition of the pullback diffeology and
that of the diffeology of $X$ must be an ordinarily smooth
$\matR$-valued map; no condition is however imposed on $p_2$.
Actually, since the pullback diffeology is the coarsest one, we
should be able to allow for $p_2$ to be \emph{any} $\matR$-valued
map.

An analogous conclusion can be drawn for any plot
$q:U\to\matR^2=V_2$, \emph{i.e.}, a plot of $\calD_V$ whose image is
contained in $V_2\subset V$. Namely, if we write
$q(u)=(q_1(u),q_2(u))$ then $q_2$ must be a usual smooth
$\matR$-valued map, while $q_1$ could be any map. The observations
thus made show that every fibre has the coarse
diffeology.\footnote{Note that $V$ thus being a diffeological vector
space over $X$, namely, carrying a diffeology with respect to which
the addition and scalar multiplication on single fibres are smooth,
is just a consequence of all fibres having coarse subset diffeology
(the coarse diffeology is automatically a vector space diffeology,
for any vector space structure).}
\end{example}

The conclusion drawn in the example is a consequence of endowing the
total space $V$ with the pullback diffeology, which by definition is
the coarsest diffeology such that the projection is smooth. Such a
diffeology is in general too big; in fact, it is reasonable to at
least restrict ourselves to continuous, in the ordinary sense, maps
(mostly because we wish to preserve the existing topology of the
spaces under consideration). Nonetheless, for the moment it serves
us to illustrate the \emph{a priori} extension of the concept of a
diffeological vector space over a given base.

\section{Particular cases of vector spaces and vector ``bundles''}

Before turning to our main subject, we examine here several specific
examples, first of vector spaces, then of the diffeological
(counterpart of) vector ``bundles'',\footnote{We put the quotations
marks, since, as we have seen already and are about to show in more
detail, frequently they are not bundles in the usual sense; rightly
so, since treating such objects, in a manner consistent with the
standard case, is among the aims of diffeology.} which, although
simple, are peculiar to diffeology. We have already given two of
such examples in the previous section, as a preliminary
illustration; the further examples that we are providing now attempt
to point towards a coherent picture, starting from ones that are
probably always expected to be found among the most basic
constructions (they will also serve in the later sections to
illustrate the constructions carried out therein).

\subsection{The choice of terminology}

In the rest of the paper we opt for the term \textbf{diffeological
vector pseudo-bundle} to denote the same object that is called a
\emph{regular vector bundle} in \cite{vincent} and a
\emph{diffeological vector space over $X$} in \cite{CWtangent} (the
definition of which we cited in the previous section). We avoid the
former term to distinguish our objects of interest from true
diffeological vector bundles (that are locally trivial), while the
term of Christensen-Wu can be confused with a diffeological vector
space \emph{proper} (that is, a vector space endowed with a vector
space diffeology); besides, it requires to introduce a notation for
the base space, something which on occasion might be superfluous or
cumbersome. The choice that we favour, that of term
\emph{pseudo-bundle}, also underlines the fact in many natural
examples (although that is by no means necessary) are objects are
indeed true vector bundles outside of, say, zero measure subset (of
the base).

\subsection{Examples of vector spaces}

We start by providing several examples of finite-dimensional
diffeological vector spaces that do not have the standard diffeology
(but whose diffeology is not particularly large; typically, we take
the finest diffeology that contains the standard one, as it must,
plus one extra map). These examples will also come into play when we
turn to consider vector pseudo-bundles.

\begin{example}\label{Rn:vector:space:nontrivial:ex}
Let $V=\matR^n$, and let $\{e_1,\ldots,e_n\}$ be its canonical
basis. Let $p:\matR\to V$ be the map acting by
$p(x)=|x|e_n$;\footnote{In place of $|x|$, we can take any function
that is not differentiable in at least one point; we take $|x|$,
since it is the easiest specific example.} endow $V$ with the finest
vector space diffeology for which $p$ is a plot. This example has
already been considered in \cite{multilinear} (see also
\cite{pseudometric}); we briefly recall that for this choice of $V$,
its diffeological dual $V^*$ has dimension $n-1$, as one can see
writing an arbitrary element of $V^*$ as $\sum_{i=1}^n a_ie^i$.
Indeed, taking the composition of this sum with $p$, one obtains the
map $x\mapsto a_n|x|$; this needs to be an ordinary smooth map
$\matR\to\matR$, which implies that $a_n=0$. Furthermore, the
diffeology of $V$ is obviously not a standard one, which, as has
already been mentioned, implies that $V$ does not admit a smooth
scalar product. This can easily be seen directly: if $A$ is an
$n\times n$ symmetric matrix that defines a smooth bilinear form on
$V$ than composing this form with the plot $(c_v,p)$ of the product
$V\times V$, where $v\in V$ is an arbitrary vector and $c_v:\matR\to
V$ is the constant map $c_v(x)\equiv v$, one sees that $e_n$ is an
eigenvector of $A$ with eigenvalue $0$.
\end{example}

The example just described is a kind of basic example for us; we
choose it as a simplest possible instance of a diffeological vector
space that carries a non-standard diffeologies. It is this example
which we will turn to most frequently (usually for such-and-such
fixed $n$) when we need to illustrate some construction that be
specific to diffeology.

There is a variation on this example, which in fact is only
different in appearance;\footnote{The difference is in the choice of
a non-canonical basis.} nevertheless, we describe it for
illustrative purposes.

\begin{example}\label{R3:vector:space:nontrivial:sum:ex}
Let $V=\matR^3$; endow it with the finest vector space diffeology
generated by the map $p:\matR\to V$ acting by $p(x)=(0,|x|,|x|)$.
The space we obtain is quite similar to the previous example; in
fact, it becomes precisely the same if, instead of taking the
canonical basis, we take any other basis where the third vector is
the vector $e_2+e_3$. On the other hand, this example allows to
illustrate easily that for diffeological vector spaces there exists
a difference between smooth and non-smooth decompositions into
direct sums. Namely, if the underlying vector space $V$ decomposes
into a direct sum of two of its vector subspaces, the corresponding
direct sum diffeology on $V$ obtained from the subset diffeologies
on $V_1$ and $V_2$ may be finer than the initial diffeology of $V$.
This is precisely the case for $V$ in this example (see
\cite{pseudometric} for details), if we take
$V_1=\mbox{Span}(e_1,e_2)$ and $V_2=\mbox{Span}(e_3)$; for both of
these the subset diffeology is the standard one, and therefore so is
the sum diffeology of their direct sum. On the other hand, the
initially chosen diffeology on $V$ is obviously not the standard
one.
\end{example}

Our third example is different from the previous two in that it has
a kind of two-dimensional nature; we will use it to illustrate some
``non-splitting'' (in the purely diffeological sense) properties.

\begin{example}\label{R2:vector:space:angled:ex}
Let now $V=\matR^2$; consider a map $p:\matR\to V$ given by the
following rule: $p(x)=(x,0)$ for $x\geqslant 0$ and $p(x)=(0,|x|)$
for $x<0$. Let $\calD_V$ be the finest vector space diffeology on
$V$ generated by the map $p$.\footnote{We observe, for future use,
that the subset $X=\{(x,y)\,|\,xy=0\}$ (once again, the union of the
coordinate axes) considered with the corresponding subset diffeology
is different from the same set considered with the subset diffeology
relative to the standard one on $\matR^2$.} The differences of this
space with respect to the standard $\matR^2$ are similar to those of
the (instance for $n=2$ of) space that appears in the previous
example. Specifically, if $f=(a_1\,\,a_2)$ is a smooth linear map
then composing it with $p$ we get $(f\circ p)(x)=a_1x$ for
$x\geqslant 0$ and $(f\circ p)(x)=a_2|x|=-a_2x$ for $x<0$, which
implies that for such a map to be smooth in the usual sense we must
have $a_2=-a_1$, so once again the diffeological dual has dimension
one (it is generated by the map $e^1-e^2$). Based on a result in
\cite{multilinear}, we conclude that any smooth bilinear form on
such a $V$ is a multiple of $(e^1-e^2)\otimes(e^1-e^2)$ by a smooth
real function; it follows that it must be degenerate (the vector
$e_1+e_2$ belonging to the kernel of the corresponding quadratic
form).
\end{example}

The three diffeological vector spaces\footnote{One of which is in
fact a family of spaces, but they differ by dimension only, so
informally we refer to them as a ``single space''.} thus described
will provide the main building blocks for our further constructions.

\subsection{Diffeological vector pseudo-bundles over finite-dimensional
diffeological vector spaces}

Let us now turn to the case of a diffeological vector pseudo-bundle
$\pi:V\to X$ over a finite-dimensional diffeological vector space
$X$, with fibres of finite dimension as well. We assume the
underlying vector space of $X$ to be identified with $\matR^k$, for
appropriate $k$, and we assume that the underlying map between
topological spaces is a true bundle (which is the simplest case,
obviously). Unless specified otherwise, we denote the diffeologies
on $V$ and on $X$ by $\calD_V$ and $\calD_X$ respectively.

\subsubsection{Vector space diffeology and vector pseudo-bundle diffeology
on $\matR^n\to\matR^k$}

It stems immediately from the above paragraph that the pseudo-bundle
we are to consider are, from the topological point of
view,\footnote{Formally we should say that their image in the
category of topological vector spaces is the projection of form
$(x_1,\ldots,x_n)\mapsto(x_1,\ldots,x_k)$.} just projections of some
$\matR^n$ to some $\matR^k$; these are quite natural to treat in as
much detail as possible, both because they are the simplest ones and
because they are precisely the diffeological vector pseudo-bundles
that restrict to (said formally, whose images under the forgetful
functor into the category of topological vector spaces are) trivial
vector bundles (of finite dimension). What we need to specify at
this point is how the diffeology on the total space (some $\matR^n$)
of the pseudo-bundle under consideration is defined.

\paragraph{The vector space structure on fibres} This is something
we have already mentioned in the examples in Section 1. Typically,
the bundle we consider is defined as the projection on the first $k$
coordinates, $k<n$. Then the fibre over a given point $x$ is
identified with $\matR^{n-k}$ by taking its last $n-k$ coordinates.
This allows to pull back the vector space operations, obtaining a
natural vector space structure on each fibre. Stated more formally,
we represent $\matR^n$ as the direct product of
$\matR^k\times\matR^{n-k}$, so that each fibre has form
$\{x\}\times\matR^{n-k}$ and carries the vector space structure of
the second factor.

\paragraph{The projection $\matR^n\to\matR^k$ and diffeology of
$\matR^n$} What is said in the previous paragraph is standard, and
allows for a rather obvious construction of a vector pseudo-bundle
diffeology on a given $\matR^n$, writing, again,
$\matR^n=\matR^k\times\matR^{n-k}$ and taking the product diffeology
coming from whatever diffeology the base space ($\matR^k$) carries
and any vector space diffeology on $\matR^{n-k}$ (see more details
on this below). What this gives however is a trivial bundle, not
only from the topological, but also from the diffeological point of
view; in order to have more intricate examples we need to discuss
some constructions specific to diffeology.

The one construction that we will frequently need is that of the
diffeology generated by a given set of maps; and, for our examples,
this, pretty much always, means the diffeology generated by a single
map. What we need to discuss here is what this means for diffeology
on $\matR^n$ that would allow to consider it as a vector
pseudo-bundle with respect to the projection $\pi$ on, say, the
first $k$ coordinates. Namely, fixed some positive integer $n$ and a
map $p:U\to\matR^n$, recall that the diffeology generated by $p$ is
the diffeology that consists of maps that locally either are
constants or filter through $p$; it is quite evident that such a
diffeology may not give any vector space structure on the fibres of
our prospective pseudo-bundle.\footnote{Let us see an example of
this. Consider $\matR^2\to\matR$, the projection of $\matR^2$ onto
its $x$-axis; let $p:\matR\to\matR^2$ be given by $p(x)=(x,x)$, and
denote by $\calD_p$ the diffeology generated by $p$. Let
$L_{x_0}=\{(x_0,y)|y\in\matR\}$ be any vertical line; then the
subset diffeology on $L_{x_0}$ relative to $\calD_p$ includes
obviously only (locally) constant maps. Indeed, let $q:U\to L_{x_0}$
be a plot of this subset diffeology; if it is not constant then the
map $u\mapsto (x_0,q(u))$ filters through $p$ (assuming $U$ is small
enough), that is, there is a smooth $f:U\to\matR$ such that for all
$u\in U$ we have $(x_0,q(u))=(f(u),f(u))$, that is, $q(u)\equiv x_0$
and so it is a constant map, after all. Finally, recall that the
diffeology that consists of locally constant maps only is not a
vector space diffeology (this has already been mentioned in the
first section), not being closed under the products by smooth
maps.}\footnote{Note that we make no reference to the diffeology on
the base space, namely, we do not consider the question of the
projection being smooth, since our aim at the moment is to point out
the absence of the structure of diffeological vector space on the
fibres (in any case, if the base space is not fixed from the
diffeological point of view, it can always be endowed with the
pushforward diffeology).}

The observation in the previous paragraph is quite evident; let us
now consider a slightly trickier question. Denote by $\calD_p^l$ the
vector space diffeology generated by $p$; this, by definition,
consists of all finite linear combinations, with smooth functional
coefficients, of the plots of $\calD_p$. Let us consider the
question of whether $\calD_p^l$ is necessarily a vector
pseudo-bundle diffeology for the projection $\pi$ (the answer
depends, obviously, on the choice of $p$).

Let us consider the specific map $p$ described in the footnote to
the previous paragraph. A plot $U\to\matR^2$ of the diffeology
$\calD_p^l$ writes as $u\mapsto F(u)+\sum_{i=1}^m(f_i(u),f_i(u))$,
where $F$ is any smooth (in the usual sense) $\matR^2$-valued
function. Precisely because it is arbitrary, and all $f_i$'s are
smooth as well, we conclude that $\calD_p^l$ is the standard
diffeology of $\matR^2$ (and in particular is a pseudo-bundle
diffeology); and this conclusion does not depend on the specific
choice of $p$, but only on the fact that it is, in turn, a smooth
function.

Thus, to obtain an example of a substantially different kind. we
actually need to choose for $p$ some (for instance)
non-differentiable function. Let us consider the following example
(it is given by one of the simplest non-differentiable functions
$\matR^2\to\matR^2$, the range being the prospective pseudo-bundle).

\begin{example}\label{vector:space:not:bundle:ex}
Let $V=\matR^2$ endowed with the (finest) vector space diffeology
generated by the plot $p:\matR^2\to V$ acting by $p(x,y)=(x,|y|)$;
if $\pi$ is the projection of $V$ onto its first coordinate then
$\pi\circ p$ is obviously smooth for the standard diffeology of
$\matR$. Let us first determine the subset diffeology of a generic
fibre; let $x_0\in\matR$. The fibre $\pi^{-1}(x_0)$ is the set
$Y_0=\{(x_0,y)|y\in\matR\}$; we claim, first of all, that its
diffeology includes the plot $q:\matR\to Y_0$ acting by
$q(y)=(x_0,|y|)$. To show that this is a plot for the subset
diffeology of $Y_0$ it is sufficient to write it as the composition
of $p$ with a usual smooth function; it suffices to take
$f:\matR^2\to\matR^2$ defined by $f(x,y)=(x_0,y)$ (obviously smooth)
to get $(p\circ f)(x,y)=p(x_0,y)=(x_0,|y|)=q(y)$.\footnote{It is
obvious that this reasoning is quite general. Namely, if we wish to
endow $\matR^n$ with a diffeology such that the subset diffeology on
fibres $\{x\}\times\matR^{n-k}$ include a given plot
$q:U\to\matR^{n-k}$, it is sufficient to endow $\matR^n$ with a
diffeology containing a plot $p:\matR^k\times U\to\matR^n$ acting by
$p(x,y)=(x,q(y))$ for $x\in\matR^k$ and $y\in U$ (a lot of our
examples are of this kind). A further extension of this idea could
be to define $p(x,y)=(x,x_1q(y))$; this gives a diffeology which is
not a direct product, namely, it is standard over the subspace
$\{0\}\times\matR^{n-1}$ and, unless $q$ smooth, non-standard over
the remaining points. See more on this last construction below.}
Thus, the subset diffeology on $Y_0$ contains the diffeology
generated by $q$; but \emph{a priori} it is a diffeology of a
diffeological space, not necessarily a vector space diffeology.

Consider now the question of whether the finest vector space
diffeology on $\matR^2$ generated by $p$ makes $\pi:V\to\matR$ into
a vector pseudo-bundle, or, in other words, if the above subset
diffeology on $Y_0$ is a vector space diffeology. To answer this
question, it is sufficient note that the addition map, which is
smooth for the vector space diffeology of $V$, is not the same as
the addition on fibres. Indeed, when we take two points $(x,y_1)$
and $(x,y_2)$ in the same fibre, the former yields $(2x,y_1+y_2)$,
the latter, $(x,y_1+y_2)$, and in particular, the composition with
the plot $p$ in the former case is $(2x,|y_1|+|y_2|)$ and in the
latter, $(x,|y_1|+|y_2|)$; so in the neighbourhood of a point
$(x,0)$ with $x\neq 0$ the smoothness of the former does not
guarantee the smoothness of the latter.
\end{example}

\paragraph{The pseudo-bundle diffeology on $\matR^n$ generated by
a given plot} As follows from the discussion in the previous
paragraph, in order to obtain a diffeological vector pseudo-bundle
$\matR^n\to\matR^k$ which overlies the natural projection, starting
from a given plot (or a family of plots; the essence does not change
much under such a generalization), we need to introduce a separate
notion (even if it is not particularly new; it is based on the
existing notions). We give it as follows.

\begin{defn}
Let $X$ be a diffeological space, and let $\pi:V\to X$ be a
surjective map defined on a set $V$ such that for every $x\in X$ the
pre-image $\pi^{-1}(x)$ has a vector space structure. Let
$A=\{p_i:U_i\to V\}_{i\in I}$ be a collection of maps, each defined
on a domain $U_i$ of some $\matR^{m_i}$, and such that $\pi\circ
p_i$ is a plot of $X$ for all $i\in I$. Let $\calD$ be the
diffeology on $V$ generated by $A$; the \textbf{pseudo-bundle
diffeology on $V$ generated by $A$} is the smallest diffeology that
contains $\calD$ and that makes the fibrewise addition and scalar
multiplication on $V$ smooth.
\end{defn}

In other words, the pseudo-bundle diffeology generated by $A$ is the
smallest diffeology containing $A$ and that makes $\pi:V\to X$ into
a diffeological vector pseudo-bundle. Note that, with the way this
definition is stated, we need to explain why it makes sense; more
precisely, why the pseudo-bundle diffeology exists. It does for
essentially the same reason (having to do with the lattice property
of diffeologies, see \cite{iglesiasBook}, Section 1.25), which is
already explained in \cite{CWtangent} (see Proposition 4.6, cited
also in the present paper, Sect. 3.1). We state this formally.

\begin{lemma}\label{pseudo:bundle:diff-gy:exists:lem}
For every $\pi:V\to X$ and $A$ as above, the pseudo-bundle
diffeology exists and is unique.
\end{lemma}

\begin{proof}
It suffices to consider $V$ as a diffeological space endowed with
the diffeology $\calD$; then the hypotheses of Proposition 4.6 of
\cite{CWtangent} are satisifed, and the proposition affirms
precisely the existence and uniqueness of the pseudo-bundle
diffeology (although it goes under a different name therein).
\end{proof}

\begin{rem}
The construction of the pseudo-bundle diffeology whose existence is
guaranteed by the proposition of \cite{CWtangent} (see Lemma
\ref{pseudo:bundle:diff-gy:exists:lem}) above is rather evident,
especially for the type of the examples that we wish to consider;
let us outline it for the case of a single generating map (for
simplicity, we will restrict ourselves to these types of examples).
So, suppose we have the projection $\pi:\matR^n\to\matR^k$ of
$\matR^n$ onto its first $k$ coordinates; let us consider the
pseudo-bundle diffeology on $\matR^n$ generated by a certain map
$p:U\to\matR^n$. Write the map $p$ in the usual coordinates of
$\matR^n$, that is, as $p(u)=(p_1,...,p_n)$. Then the necessary
condition for the existence of the pseudo-bundle diffeology
generated by $p$ is that the map
$\hat{p}_k=(p_1,...,p_k):U\to\matR^k$ be a plot of whatever
diffeology we choose to endow the base space $\matR^k$ with. Now,
denote by $\hat{p}_{n-k}:U\to\matR^{n-k}$ the map
$(p_{k+1},\ldots,p_n)$. Then the pseudo-bundle diffeology on
$\matR^n$, relative to the projection $\pi$, is the diffeology
generated by the collection of all maps of the following form:
$u\mapsto(\hat{p}_k(u),p_{n-k}'(u))$ where $p_{n-k}'$ belongs to the
subset of maps defined on $U$, of the vector space diffeology
generated by $\hat{p}_{n-k}$.
\end{rem}

\subsubsection{The diffeology on the base space}

Here we make some remarks concerning the choice of diffeology on the
base space of a (prospective) diffeological vector pseudo-bundle;
while it is most natural to consider the base space as (an
assembling of some copies of) the standard $\matR^k$, this is not a
necessity \emph{a priori}, so we add some precise comments to the
matter.

\paragraph{Non-standard diffeology on the base space} As we have
already done in some examples, and as we will continue to do below,
we frequently (but, of course, not always) consider pseudo-bundles
whose underlying topological structure is that of a standard
projection of $\matR^n$ onto its first $k$ coordinates; and we
typically endow $\matR^k$ with the standard diffeology and $\matR^n$
with the (finest) vector space diffeology generated by a specified
plot. These choices do pose a couple of questions, which we do not
wish to consider in much detail, but it is worthwhile to say a few
words about them. Namely, a minor question is that of diffeology on
the base space $\matR^k$; it certainly does not have to be the
standard diffeology, and other choices are available, which give the
same underlying topological structure, including the
almost-classical alternative of \emph{wire diffeology} (see
\cite{iglesiasBook}, Section 1.10). The other question regards
$\matR^n$, which, for the projection indicated, must carry a
diffeology that yields a vector space diffeology on the last $n-k$
generators; but that does not automatically imply that it must be a
vector space diffeology on the whole of $\matR^n$.

\paragraph{The wire diffeology on $\matR^k$} Let us assume that
$\matR^k$ is endowed with the wire diffeology, \emph{i.e.}, the
diffeology generated by all smooth maps $\matR\to\matR^k$ (recall
that by Theorem 3.7 of \cite{CSW_Dtopology} this implies that the
underlying D-topology is the usual topology of $\matR^k$). We wonder
what conditions this imposes on the diffeology of $\matR^n$, so that
the projection of $\matR^n=\matR^k\times\matR^{n-k}$ (in the sense
of the usual vector spaces, not diffeological ones) onto the first
factor be smooth.

Using the usual coordinates, we write an arbitrary plot
$p:U\to\matR^n$ as $p(u)=(p_1(u),\ldots,p_n(u))$; then $(\pi\circ
p)(u)=(p_1(u),\ldots,p_k(u))$, and this must be a plot for the wire
diffeology on $\matR^k$ (assuming that $p$ is not constant, and that
$U$ is small enough, this means that $\pi\circ p$ filters through a
smooth function of one variable). This means that the subset
diffeology on $\matR^k\subset V$ is contained in the wire diffeology
of $\matR^k$ (and it does not have to coincide with it; it might in
fact be strictly smaller, for instance, it could be the discrete
diffeology, unless we impose topological restrictions which would
prevent it from being so). This is the best general conclusion that
could be reached under our assumptions.

\paragraph{Vector pseudo-bundle diffeology on $\matR^n$ without vector
space structure} As follows from the observations in the previous
paragraph, a relevant example is quite easy to construct; it
suffices to take $k\geqslant 2$, $n>k$, and endow $\matR^k$ with the
aforementioned wire diffeology. After that, present $\matR^n$ as the
direct product $\matR^k\times\matR^{n-k}$ and endow it with the
product diffeology coming from the wire diffeology on $\matR^k$ and
any vector space diffeology (for instance, the standard one) on
$\matR^{n-k}$; the projection onto the first factor of the direct
product is a diffeological vector pseudo-bundle (a true bundle,
actually). As has already been mentioned in Section 1, the wire
diffeology is not a vector space diffeology, so we get an example of
the kind described in the title of the paragraph.

\subsubsection{The underlying topological bundle}

We collect here some rather simple remarks, which point to a
classification of the simplest possible diffeological vector
pseudo-bundles, although not in an exhaustive manner. The main point
we make here is that, if our base space is $\matR^k$ with a
diffeology such that the induced topology (the so-called D-topology
\cite{CSW_Dtopology}) is the usual one, then the topological bundle
underlying a diffeological vector pseudo-bundle over it is
(obviously; this is well-known) trivial, but it might well not be so
from the diffeological point of view.

\paragraph{Topologically and diffeologically trivial bundles} This
is the simplest case, corresponding to a usual trivial bundle; we
have $\pi:V\to X$ with $X=\matR^k$ for an appropriate $k$ and $V$ is
diffeomorphic to $X\times V'$ where $V'$ is a finite-dimensional
diffeological vector space, with $\pi$ being just the projection of
the direct product on the first factor. Identifying the underlying
vector space of $V'$ with $\matR^m$ for an appropriate $m$ and
denoting its diffeology by $\calD_m'$ (this is a vector space
diffeology on $\matR^m$, which can be characterized as containing
all usually smooth maps with values in $\matR^m$ and being closed
under linear combinations with functional coefficients), we obtain
an identification of $V$ with $\matR^{m+k}$ with diffeology that (in
the obvious sense) splits as $\calD_X\times\calD_m'$.

\paragraph{Topologically trivial but diffeologically non locally-trivial
pseudo-bundles} The question of the existence of objects of which
the title of this paragraph speaks can be re-stated as, does the
diffeology $\calD_V$ always ``decomposes'' as
$\calD_X\times\calD_m$, for some vector space diffeology $\calD_m$
on $\matR^m$, or does there exist a diffeological vector bundle
$\pi:V\to X$ (\emph{i.e.}, diffeologies $\calD_V$ and $\calD_X$ on
$\matR^{m+k}$ and $\matR^k$ respectively) such that the underlying
topological vector bundle is the projection of $\matR^{m+k}$ onto
its first $k$ coordinates, but which is not trivial as a
diffeological vector bundle, \emph{i.e.} such that its fibres have
different diffeologies? The answer is positive to the latter
question (and negative to the former), as the following example
shows.

\begin{example}\label{top:trivial:dif:nontrivial:ex}
Let $V$ be $\matR^2$ endowed with the pseudo-bundle diffeology
generated by the map $p:\matR^2\to V$ given by $p(x,y)=(x,|xy|)$,
and let $X$ be $\matR$, identified with the $x$-axis of $V$ and
endowed with the standard diffeology. Let $\pi:V\to X$ be the
projection onto the first coordinate; it is obviously smooth, and
gives a diffeological vector bundle. However, the diffeology on the
fibre at the origin is the standard one, while elsewhere it is not:
for $x\neq 0$ the fibre $\pi^{-1}(x)$ carries the vector space
diffeology generated by the map $y\mapsto\mbox{const}\cdot|y|$ (with
non-zero constant), which is strictly coarser than the standard
one.\footnote{Once again, we note that the first coordinate of $p$
could be any smooth function.}
\end{example}

Informally speaking, the main conclusion that we draw from this
example is that the condition that a diffeological vector
pseudo-bundle be trivial should be imposed separately, it being
independent from other assumptions. Another point that we stress
here is that the pseudo-bundle $\pi:V\to X$ of Example
\ref{top:trivial:dif:nontrivial:ex} is not even not locally trivial
from the diffeological point of view (there is not a neighbourhood
of $0\in\matR$ where the diffeology on the pseudo-bundle be a direct
product diffeology). Since this example can be easily extended to
any other pair $(\matR^n,\matR^k)$, we state separately the
following:

\begin{oss}
For every pair of natural numbers $k<n$, there exists a
diffeological structure on $\matR^n$ and a smooth projection
$\pi:\matR^n\to\matR^k$ such that $\pi$ is a diffeological vector
pseudo-bundle that is locally non-trivial in at least one point.
\end{oss}

\paragraph{Topologically trivial and diffeologically locally
trivial non-trivial} This is the final \emph{a priori} possibility
to consider. We'll leave the question of existence of such in the
open (although the answer is probably negative).

\subsubsection{Under the passage to duals}

We discuss the duals (see \cite{vincent} for the definition) of
diffeological vector pseudo-bundles in more detail later on;
however, since the passage to duals illustrates particularly well
the peculiarities of diffeological pseudo-bundles (as opposed to the
usual vector bundles), we put here a preliminary example to the
matter. Let us consider again the Example
\ref{top:trivial:dif:nontrivial:ex}, \emph{i.e.}, the projection on
the $x$-axis $\pi:\matR^2\to\matR$, where the diffeology on
$\matR^2$ is generated by the plot $p(x,y)=(x,|xy|)$. We have
observed that the map $\pi$, which obviously defines a topologically
trivial vector bundle, gives a non-trivial vector pseudo-bundle from
the diffeological point of view, since the fibre at zero is not
diffeomorphic as a diffeological vector space to any other fibre. We
now observe that this has visible implications if we consider the
(diffeological) dual bundle $\pi^*:(\matR^2)^*\to(\matR)^*$.

Note first of all that if $\matR\ni x\neq 0$ then $\pi^{-1}(x)$ is
an instance (with $n=1$) of a vector space considered in Example
\ref{Rn:vector:space:nontrivial:ex}. It was already mentioned there
that the diffeological dual of such a space has strictly smaller
dimension than the space itself, which in the case we are treating
now implies that the dual is just the zero space.\footnote{This is
also easy to see directly.} On the other hand, the fibre
$\pi^{-1}(0)$ has standard diffeology, thus its diffeological dual
is $\matR$ with the standard diffeology. This easily implies that,
from the topological point of view, the total space $(\matR^2)^*$ of
the dual bundle is the wedge of two copies of $\matR$ joined at the
origin. Furthermore, its diffeology is equivalent to the subset
diffeology of the union of coordinate axes in the standard $\matR^2$
(to which $(\matR^2)^*$ is naturally identified).

\section{Constructing diffeological vector pseudo-bundles}

In this section we discuss some issues related to constructing
diffeological vector pseudo-bundles, starting with recalling briefly
the way described by Christensen-Wu that allows to obtain a vector
pseudo-bundle given a smooth surjection $\pi:V\to X$ such that all
fibres carry a vector space structure but not necessarily that of a
diffeological vector space. We then consider a kind of diffeological
gluing of vector (pseudo-)bundles, as a means of obtaining new
pseudo-bundles, in particular, non locally trivial ones. This
constructive approach of obtaining, by successive gluings of simpler
pseudo-bundles (ideally those treated in the preceding section,
\emph{i.e.}, with the underlying topological bundles trivial), could
be seen as a way to partially compensate for the absence of local
trivializations; but we say right away that it is only partial (we
show in a section below that there are pseudo-bundles not admitting
such a decomposition, for diffeological reasons).

\subsection{Obtaining the structure of a diffeological vector space
on fibres}

One situation that might easily present itself when trying to
construct a specific diffeological vector pseudo-bundle is that at
some point we get a kind of vector space \emph{pre-bundle} (in the
terminology of \cite{vincent}), that is, one where the subset
diffeology on fibres is actually finer than a vector space
diffeology (\emph{i.e.}, the fibres are vector spaces, but the
addition and/or scalar multiplication are not smooth in general).
That this can actually happen is demonstrated, once again, by the
Example 4.3 of \cite{CWtangent}: the internal tangent bundle of the
coordinate axes in $\matR^2$ considered with the Hector's diffeology
(see Definition 4.1 of \cite{CWtangent} for details). As is shown in
\cite{CWtangent}, the tangent space at the origin is \emph{not} a
diffeological vector space for the subset diffeology.\footnote{The
existence of such examples motivates the introduction of the dvs
diffeology on internal tangent bundles by the authors.}

In both \cite{vincent} (Theorem 5.1.6) and \cite{CWtangent}
(Proposition 4.16), it is shown that the diffeology on the total
space can be ``expanded'' to obtain a diffeological vector space
bundle. We now cite the latter result (recall that the term
``diffeological vector space over...'' means the same thing as our
term ``diffeological vector pseudo-bundle'').

\begin{prop} \emph{(\cite{CWtangent}, Proposition 4.6)}
Let $\pi:V\to X$ be a smooth map between diffeological spaces, and
suppose that each fibre of $\pi$ has a vector space structure. Then
there is a smallest diffeology $\calD$ on $V$ which contains the
given diffeology and which makes $V$ into a diffeological vector
space over $X$.
\end{prop}

The diffeology whose existence is affirmed in this proposition can
be described explicitly (see \cite{CWtangent}, Remark 4.7). It is
the diffeology generated by the linear combinations of plots of $V$
and the composite of the zero section with plots of $X$. More
precisely, a map $\matR^m\ni U\to V$ is a plot of $\calD$ if and
only if it is locally of form $u\mapsto
r_1(u)q_1(u)+\ldots+r_k(u)q_k(u)$, where $r_1,\ldots,r_k:U\to\matR$
are usual smooth functions (plots for the standard diffeology on
$\matR$) and $q_1,\ldots,q_k:U\to V$ are plots for the pre-existing
diffeology of $V$ such that there is a single plot $p:U\to X$ of $X$
for which $\pi\circ q_i=p$ for all $i$.

\subsection{Diffeological gluing of vector pseudo-bundles}

We now give a precise description of a construction that allows to
obtain a wealth of examples of diffeological vector pseudo-bundles
by ``piecing together'' (or ``assembling'') some simpler
pseudo-bundles. This construction, which we call \emph{diffeological
gluing}, is essentially an extension to the present context of the
usual topological gluing. It also mimicks the ``assembly'' of the
usual vector bundles over smooth manifolds from the trivial bundles
over an appropriate $\matR^n$ (however, it does not possess the same
universality property, meaning that there are plenty of
diffeological vector bundles that are not obtained by gluing; see
the section that follows for more details on this aspect).

\paragraph{Gluing of diffeological spaces} Suppose that we have two
diffeological vector pseudo-bundles $\pi_1:V_1\to X_1$ and
$\pi_2:V_2\to X_2$. We wish to describe a gluing operation on these,
that would give us again a diffeological vector pseudo-bundles. This
obviously requires an appropriate gluing operation on diffeological
spaces (applied twice, to the pair of the base spaces and to the
pair of the total spaces, in the latter case with some further
restrictions to preserve the linear structure). This, in turn,
requires a smooth map $f:X_1\supset Y\to X_2$ and its lift to a
smooth map $\tilde{f}:\pi_1^{-1}(Y)\to V_2$; this lift should be
linear on the fibres (this is the above-mentioned further
restriction).

For the sake of clarity, we comment right away how this construction
relates to the example of the coordinate axes in $\matR^2$. It is
not meant to produce it immediately; rather, it describes the first
step in the construction, by setting $X_1$ one of the axes with its
subset diffeology and the corresponding internal tangent bundle
(which is the usual tangent bundle to $\matR$), the subset $Y$ is
the origin, and finally $X_2$ is a single point and the
corresponding bundle is the map $\pi_2$ that sends (another copy of)
$\matR^2$, with the standard diffeology, to this point. The map $f$
is obvious and sends the origin $(0,0)\in\matR^2$ to the point that
composes $X_2$. As far as the lift $\tilde{f}$ is concerned, there
are numerous choices. Indeed, this lift is defined on the line
$\matR=\pi_1^{-1}(0,0)$, so it is a linear map from $\matR$ to
$\matR^2=\pi_2^{-1}(X_2)$, thus there are infinite possibilities (we
need to specify first the image, which is either the origin or any
$1$-dimensional linear subspace of $\matR^2$, and in this latter
case we should specify the action of the map
$\tilde{f}$).\footnote{Since the diffeology on
$\matR=\pi_1^{-1}(0,0)$ is standard, it suffices that it be a linear
map; its smoothness then is automatic.}

What has just been said is sufficient to describe the prospective
gluing from the topological point of view; the diffeological aspect
can be easily obtained by employing the concept of the quotient
diffeology, something that we specify in the paragraph that follows.
Assuming this has been done, what we obtain is a map $\pi:V\to X$,
where $V$ is the result of gluing $V_1$ and $V_2$ along the map
$\tilde{f}$, $X$ is the result of gluing $X_1$ and $X_2$ along the
map $f$, and $\pi$ is induced by $\pi_1$ and $\pi_2$ in the obvious
way. In what follows we will show that this map is indeed a
diffeological vector pseudo-bundle, under the assumption that the
map $f$ is injective.

\paragraph{Diffeology on an assembled space} We use the term ``assembled
space'' to refer to the space resulting from gluing together two
other spaces. Suppose we have two diffeological spaces $X_1$ and
$X_2$ that are glued together along some smooth map $f:X_1\supset
Y\to X_2$ (the smoothness of $f$ is with respect to the subset
diffeology of $Y$). Recall first the standard (topological)
definition of gluing, $X_1\cup_f X_2=(X_1\sqcup X_2)/\sim$, where
$\sim$ is the following equivalence relation: $x_1\sim x_1$ if
$x_1\in X_1\setminus Y$, $x_2\sim x_2$ if $x_2\in X_2\setminus
f(Y)$, and $x_1\sim x_2$ if $x_1\in X_1$ and $x_2=f(x_1)$.

Now, for $X_1$ and $X_2$ diffeological spaces, there is a natural
diffeology on $X_1\sqcup X_2$, namely the sum diffeology; and for
whatever equivalence relation exists on a diffeological space (which
is $X_1\sqcup X_2$ in this case) there is the standard quotient
diffeology on the quotient space. This is the \emph{gluing}
diffeology on $X_1\cup_f X_2$.

\begin{example}\label{axes:by:gluing:wire:ex}
Let us work out an example that is close to the setting of our
interest. Take $X_1\subset\matR^2$ the $x$-coordinate axis and
$X_2\subset\matR^2$ the $y$-coordinate axis; both are considered
with the subset diffeology of $\matR^2$ (so it is the standard
diffeology of $\matR$). Gluing them at the origin yields, from the
topological point of view, the same space that appears in the
Christensen-Wu example. The question that we wish to answer now is
the following: is its gluing diffeology, as has just been described,
the same as the subset diffeology of $\matR^2$?

The answer is obviously positive in the neighbourhood of any point
that is not the origin, while in a neighbourhood of the latter the
plots of the gluing diffeology are precisely those maps that lift to
a plot of either $X_1$ or $X_2$ (but not both, which is a crucial
point). This means that a plot $p$ of gluing diffeology writes
either as $p(u)=(p_1(u),0)$ or $p(u)=(0,p_2(u))$, where $p_1,p_2$
are $\matR$-valued smooth maps (possibly zero maps). This implies
that $p$ is indeed a restriction of a smooth $\matR^2$-valued map,
hence the gluing diffeology does coincide for this space with the
subset diffeology (every restriction of a smooth map $U\to\matR^2$
is obviously of this form).
\end{example}

\paragraph{Plots of gluing diffeology} In the next paragraph we will
make use of possibly dubious notation $p_1\sqcup p_2$ to denote
plots of the disjoint union $X_1\sqcup X_2$, so we need to explain
first what we mean. Here $p_1$ and $p_2$ are plots of $X_1$ and
$X_2$ respectively; when we say that $p_1\sqcup p_2$ is a plot of
$X_1\sqcup X_2$, this means that one of these is taken into
consideration, and the other is an ``empty'' map.

\paragraph{Gluing between diffeological vector pseudo-bundles} Let us
finally consider the gluing of two diffeological vector
pseudo-bundles. Let $X_1$ and $X_2$ be two diffeological spaces, and
let $\pi_1:V_1\to X_1$ and $\pi_2:V_2\to X_2$ be two diffeological
vector pseudo-bundles over $X_1$ and $X_2$ respectively. Let
$f:X_1\supset Y\to X_2$ be a smooth injective map, and let
$\tilde{f}:\pi_1^{-1}(Y)\to V_2$ be a smooth map that is linear on
each fibre and such that
$\pi_2\circ\tilde{f}=f\circ(\pi_1)|_{\pi_1^{-1}(Y)}$. The latter
property yields an obvious well-defined map
$$\pi:V_1\cup_{\tilde{f}}V_2\to X_1\cup_f X_2.$$

\begin{thm}\label{glued:bundles:is:bundle:thm}
The map $\pi:V_1\cup_{\tilde{f}}V_2\to X_1\cup_f X_2$ is a
diffeological vector pseudo-bundle.
\end{thm}

\begin{proof}
The items to check are: 1) that $\pi$ is smooth, 2) that the
pre-image of each point is a (diffeological) vector space, 3) that
the addition is smooth as a map
$(V_1\cup_{\tilde{f}}V_2)\times_{X_1\cup_f
X_2}(V_1\cup_{\tilde{f}}V_2)\to V_1\cup_{\tilde{f}}V_2$, 4) that the
scalar multiplication is smooth as a map
$\matR\times(V_1\cup_{\tilde{f}}V_2)\to V_1\cup_{\tilde{f}}V_2$; 5)
that the zero section $X_1\cup_f X_2\to V_1\cup_{\tilde{f}}V_2$ is
smooth. We check these conditions one-by-one.

The smoothness of $\pi$ follows directly from the constructions, but
for completeness we add details. Take a plot $p:U\to
V_1\cup_{\tilde{f}}V_2$; locally it is a composition of the form
$p=\pi_V\circ(p_1\sqcup p_2)$, where $\pi_V:V_1\sqcup V_2\to
V_1\cup_{\tilde{f}}V_2$ is the quotient projection (smooth by
definition), $p_i$ for $i=1,2$ is a plot of $V_i$, and $p_1\sqcup
p_2$ is the corresponding plot of the disjoint sum. Let also
$\pi_X:X_1\sqcup X_2\to X_1\cup_f X_2$ be the quotient projection
(once again, smooth by definition).

Observe now that there is an obvious equality (by definition of
$\pi$) $\pi\circ\pi_V=\pi_X\circ(\pi_1\sqcup p_2)$, and also that
$(\pi_1\sqcup\pi_2)\circ(p_1\sqcup p_2)$ is some plot $q$ of the
disjoint sum $X_1\sqcup X_2$. Thus,
$$\pi\circ p=\pi\circ\pi_V\circ(p_1\sqcup p_2)=\pi_X\circ(\pi_1\sqcup\pi_2)\circ(p_1\sqcup p_2)=
\pi_X\circ q,$$ \emph{i.e.}, a plot of $X_1\cup_f X_2$, since
$\pi_X$ is smooth. This establishes the smoothness of $\pi$.

The structure of a vector space on each fibre is inherited from
either $V_1$ or $V_2$. More precisely, for $x\in X_1\cup_f X_2$ and
the fibre $\pi^{-1}(x)$ it is inherited from $V_1$ if $x\in
X_1\setminus Y$, otherwise it is inherited from $V_2$. Note that by
injectivity of $\tilde{f}$ the fibre at a point $y\in Y$ is actually
$\pi_2^{-1}(y)$.

The smoothness of the zero section is established by a similar
(carried out in inverse) reasoning to the one we used to show the
smoothness of $\pi$. Namely, let $q:U\to X_1\cup_f X_2$ be a plot of
$X_1\cup_f X_2$, and let $s^0:X_1\cup_f X_2\to
V_1\cup_{\tilde{f}}V_2$ be the zero section. Observe that locally
$q$ writes as $q=\pi_X\circ(q_1\sqcup q_2)$, where $q_i$ is a plot
of $X_i$ for $i=1,2$. Furthermore, let $s_i^0:X_i\to V_i$ be the
corresponding zero section; recall that it is smooth by assumption,
so the composition $(s_1^0\sqcup s_2^0)\circ(q_1\sqcup q_2)$ is some
plot $p$ of $V_1\sqcup V_2$.

Now, it is easy to see that we have the equality
$s^0\circ\pi_X=\pi_V\circ(s_1^0\sqcup s_2^0)$. Thus, we can put
everything together, obtaining
$$s^0\circ q=s^0\circ\pi_X\circ(q_1\sqcup q_2)=\pi_V\circ(s_1^0\sqcup s_2^0)\circ(q_1\sqcup q_2)=
\pi_V\circ p,$$ which is obviously a plot of
$V_1\cup_{\tilde{f}}V_2$ by smoothness of $\pi_V$. So $s^0$ is
smooth.

Let us now consider the scalar multiplication, \emph{i.e.}, the map
$\bullet:\matR\times(V_1\cup_{\tilde{f}}V_2)\to
V_1\cup_{\tilde{f}}V_2$. Recall that (by definition of the product
diffeology) a plot of the product locally writes as $(f,p)$, where
$f:U\to\matR$ is a smooth function, and $p:U\to
V_1\cup_{\tilde{f}}V_2$ is a plot of $V_1\cup_{\tilde{f}}V_2$.
Furthermore, we can assume that $U$ is small enough, so, as before,
$p$ writes as $\pi_V\circ(p_1\sqcup p_2)$ for $p_1,p_2$ some plots
of $V_1,V_2$ respectively. Then we have
$(\bullet\circ(f,p))(u)=\pi_V(f(u)p_1(u)\sqcup f(u)p_2(u))$, and
this is smooth because $\pi_V$ is smooth, and so is the scalar
multiplication on each of $V_1$, $V_2$.

It remains to consider the addition map. Formally, this is a map
$+:(V_1\cup_{\tilde{f}}V_2)\times_X(V_1\cup_{\tilde{f}}V_2)\to
V_1\cup_{\tilde{f}}V_2$. Once again, as above, a plot $p$ of the
domain product locally writes as $p=(\pi_V,\pi_V)\circ(p_1\sqcup
p_2,p_1'\sqcup p_2')$, where the plots $p_1\sqcup p_2$ and
$p_1'\sqcup p_2'$ take values in the same fibre of $V_1\sqcup V_2$.
Hence $(+\circ p)(u)=\pi_V(p_1(u)\sqcup
p_2(u))+\pi_V(p_1'(u)+p_2'(u))$, so it is smooth because the
addition on each individual $V_i$ is smooth.
\end{proof}

\begin{rem}
The final choice to carry out the gluing along injective maps is
inspired by the classical gluing along homeomorphisms and
diffeomorphisms.
\end{rem}

\subsection{Smooth maps between pseudo-bundles and gluing}

Here we consider the following situation. First of all, let
$\pi_1:V_1\to X_1$ and $\pi_2:V_2\to X_2$ be two diffeological
vector pseudo-bundles that we will glue together along a given
$f:X_1\supset Y\to X_2$ with the fixed lift $\tilde{f}$. Let also
$\pi_1':V_1'\to X_1$ and $\pi_2':V_2'\to X_2$ be two other
diffeological vector pseudo-bundles with the same respective base
spaces. Finally, let $F_1:V_1'\to V_1$ be a smooth map linear on
fibres and such that $\pi_1'=\pi_1\circ F_1$. Similarly, let
$F_2:V_2\to V_2'$ be a smooth map linear on fibres and such that
$\pi_2=\pi_2'\circ F_2$. We define the following map:
$$\tilde{f}':(\pi_1')^{-1}(Y)\to V_2',\,\,\,\,\tilde{f}'(v_1')=F_2(\tilde{f}(F_1(v_1'))).$$
It is obvious that $\tilde{f}'$ is smooth (for the subset
diffeology) and linear on fibres, since it is a composition of maps
that enjoy these two properties. Furthermore, we easily get the
following:

\begin{lemma}
The map $\tilde{f}'$ takes values in $(\pi_2')^{-1}(f(Y))$ and is a
lift of $f$ by the maps $\pi_1'$ and $\pi_2'$.
\end{lemma}

\begin{proof}
The first part of the statement follows from the definition of
$\tilde{f}'$ and from the equalities $\pi_1'=\pi_1\circ F_1$ and
$\pi_2=\pi_2'\circ F_2$. To check the second part of the statement,
we need to verify the equality $f\circ\pi_1'=\pi_2'\circ\tilde{f}'$
on the appropriate domain of the definition (which is
$(\pi_1')^{-1}(Y)$). Now, it follows from our assumptions on $F_1$,
$F_2$ and $\tilde{f}$, and from the definition of $\tilde{f}'$ that
$$f\circ\pi_1'=f\circ\pi_1\circ F_1=\pi_2\circ\tilde{f}'\circ F_1,\mbox{ and }$$
$$\pi_2'\circ\tilde{f}'=\pi_2'\circ F_2\circ\tilde{f}\circ F_1=\pi_2\circ\tilde{f}\circ F_1,$$
so the equality does hold.
\end{proof}

The main consequence of this lemma is that the map $\tilde{f}'$ can,
in its turn, be used to carry out the gluing between the
pseudo-bundles $\pi_1':V_1'\to X_1$ and $\pi_2':V_2'\to X_2$. It can
also be extended into a more general setting of having
$\pi_1':V_1'\to X_1'$ together with a smooth map $f_1':X_1'\to X_1$
that lifts to $F_1:V_1'\to V_1$ that is a smooth and linear on
fibres, and $\pi_2':V_2'\to X_2'$ together with a smooth map
$f_2':X_2\to X_2'$ that lifts to $F_2:V_2\to V_2'$, again smooth and
linear on fibres. Then defining $f':(f_1')^{-1}(Y)\to X_2'$ by
setting $f'=f_2'\circ f\circ f_1'$ and its lift
$\tilde{f}'=F_2\circ\tilde{f}\circ F_1$ allows to carry out the
gluing of the pseudo-bundles $\pi_1':V_1'\to X_1'$ and
$\pi_2':V_2'\to X_2'$, obtaining a diffeological vector
pseudo-bundle $V_1'\cup_{\tilde{f}'}V_2'\to X_1\cup_{f'}X_2$.

\section{Which diffeological vector pseudo-bundles are the result
of gluing?}

This brief section collects some preliminary observations regarding
what should be a sort of the reverse of gluing, \emph{i.e.}, a kind
of diffeological surgery. These are indeed preliminary only; the
only precise claim that we make is that the gluing operation, as can
be expected from its definition, does not produce all possible
diffeologies, not even in very simple cases.

\paragraph{Sating the problem} As we can see from Theorem \ref{glued:bundles:is:bundle:thm},
diffeological gluing produces a diffeological vector pseudo-bundle,
and it is natural then to ask the opposite question: does the gluing
procedure allow, starting perhaps from some elementary building
blocks, to obtain all diffeological vector pseudo-bundles, at least
under some topological restrictions?\footnote{As we have already
stated elsewhere, the actual idea behind asking this question is
that of looking for an appropriate substitute for local
trivializations, such as a statement of the sort: if the base space
$X$ can be covered by copies of standard $\matR^n$'s (with variable
$n$) then any finite-dimensional diffeological vector pseudo-bundle
can be obtained by gluing of copies of (limited number of models of)
diffeological vector pseudo-bundles of the kind described in Section
2. There are ways to phrase this statement such that it becomes
quite trivial; otherwise, we are not certain that it is true and
provide only preliminary statements pointing in that direction.}

It is quite easy to show (we do so below) that the answer to this
question as stated is negative. In general terms, the reason for
this is simply that the gluing diffeology is a rather specific type
of a diffeology (and this is an expected conclusion); this starts
already at the level of a single assembled space. In this section we
give a concrete illustration of this phrase, along with some partial
remarks concerning the decomposition of whole pseudo-bundles under
some assumptions on the base space.

\paragraph{The gluing of the subset diffeologies on the base space}
The question that we consider here is the following one. Let
$(X,\calD)$ be a diffeological space, and let $X_1,X_2\subset X$ be
such that there exists a (D-continuous) map $f:X_1\supset Y\to X_2$
such that $X$ is D-homeomorphic to the topological space $X_1\cup_f
X_2$. Recall that each of $X_1$, $X_2$ carries the canonical subset
diffeology (which we denote by $\calD_1$ and $\calD_2$ respectively)
relative to $\calD$; endow $X_1\cup_f X_2$ with the diffeology
$\calD'$ that is the result of gluing of $\calD_1$ and $\calD_2$.
Does $\calD'$ necessarily coincide with $\calD$, the initial
diffeology of $X$?\footnote{This is analogous in spirit to the case
of just a diffeological vector space, which, as shown in
\cite{pseudometric}, might decompose as a direct sum of two
subspaces, but the vector space sum diffeology (relative to the
subset diffeologies on the subspaces) might be finer than the
diffeology of the space itself.} The following example shows that
the answer is \emph{a priori} negative.

\begin{example}
Consider again the space $X$ of the Example
\ref{axes:by:gluing:wire:ex}, the union of the two coordinate axes
in $\matR^2$, endowed however with another (coarser) diffeology.
Namely, we consider $X$ as a subset of $\matR^2$ that is endowed
with the diffeology $\calD_X$ of Example
\ref{R2:vector:space:angled:ex}. Let us present $X$ as the union
$X=X_1\cup X_2$, where $X_1$ is the $x$-axis of $\matR^2$ and $X_2$
is the $y$-axis; let us determine first the subset diffeologies,
$\calD_1$ and $\calD_2$ respectively, of $X_1$ and $X_2$.

A map $q:U\to X_1$ is a plot of $\calD_1$ if and only if its
composition with the inclusion map is a plot of $\calD$,
\emph{i.e.}, if the map $U\to\matR^2$ given by $u\mapsto(q(u),0)$ is
a plot of $\calD$. This means that locally it is a linear
combination, with smooth functional coefficients, of maps that
either are constant or filter through $p$ via an ordinary smooth
map. It is quite clear then, that in a neighbourhood of any point
that is distinct from the origin $q$ is just an ordinary smooth map.
Now, observe that the intersection with $X_1$ of the image of $p$ is
the positive half-line (with the endpoint at $0$), and the
intersection of this image with a neighbourhood of $0$ is a
half-interval of form $[0,\varepsilon)$ (in particular, this isn't
an open set). Suppose now that $q=p\circ f$ for a smooth $f$ on some
small domain $U$ and that the image $X'\subset X_1$ contains $0$;
then $X'\subseteq[0,\varepsilon)$, since it is contained in the
image of $p$. In fact, $X'$ can be identified with a half-interval
$[0,\varepsilon')$ for a suitable $\varepsilon'$. Since $p$ is
bijective on such half-intervals, this means that our smooth $f$
sends the domain $U$ to the set $[0,\varepsilon')$, which is
obviously impossible. This allows us to conclude that the subset
diffeology $\calD_1$ of $X_1$ is the standard one; by analogous
reasoning, so is the subset diffeology $\calD_2$ of $X_2$.

On the other hand, we have already established that the gluing of
$X_1$ and $X_2$ with standard diffeologies at the origin produces
the union of the coordinate axes of $\matR^2$ with the subset
diffeology of $\matR^2$. This is clearly not the same as our
original diffeology $\calD$ on $X$.
\end{example}

What this example illustrates is that already on the base space of a
pseudo-bundle, decomposed into a gluing of two its subspaces, the
initial diffeology might not be required from gluing.\footnote{Not a
good news, in the sense of our final purpose.} For now we limit
ourselves to providing this illustration, and avoid to add further
generic comments.

\section{Pseudo-metrics on vector pseudo-bundles}

In this section we consider other types of working with
diffeological vector pseudo-bundles. In part they are operations
typical of usual vector bundles (direct sum, tensor product, dual
bundle), in part they deal with the metrics' issues. Now, the
description of the above-listed operations has been available
previously and comes from \cite{vincent};\footnote{Along with a
precise description of various examples; a number of those given in
\cite{vincent} seem to be just references to pictures, which we do
not find satisfactory.} we provide a complete account, however, also
because we will use the details of these constructions to treat the
concept of the so-called \emph{pseudo-metric on a pseudo-bundle}
that we introduce.

\subsection{Operations on vector pseudo-bundles}

Since diffeological vector pseudo-bundles in general are not locally
trivial, we cannot use local trivializations to define these
operations (in addition to having to specify the diffeology).
However, a different description of them exists (see \cite{vincent},
Chapter 5); we recall it in some detail, with a particular attention
to showing how one goes round the obstacle of the absence of local
trivializations.

\paragraph{Sub-bundles} The definition of a sub-bundle is clear, but
there is the question why it would be well-posed in reference to the
diffeological issues. Specifically, let $\pi:V\to X$ be a
diffeological vector pseudo-bundle, and let $Z\subset V$ be a subset
of $V$. Endow $Z$ with the subset diffeology.

\begin{defn}
We say that the restriction $\pi^Z:Z\to X$ of $\pi:V\to X$ to $Z$ is
a \textbf{diffeological vector sub-bundle} (or simply a sub-bundle)
of $\pi:V\to X$ if the following condition holds: for every $x\in X$
the intersection $Z\cap\pi^{-1}(X)$ is a vector subspace of
$\pi^{-1}(X)$.\footnote{In particular, being a vector subspace means
that it is non-empty, so the map $\pi^Z$ is onto $X$.}
\end{defn}

Let us formally prove that this definition is well-posed (although
this is rather obvious and is already established in
\cite{vincent}).

\begin{lemma}\emph{(\cite{vincent})}
Any diffeological vector sub-bundle $\pi^Z:Z\to X$ of a
diffeological vector pseudo-bundle $\pi:V\to X$ is itself a
diffeological vector pseudo-bundle.
\end{lemma}

\begin{proof}
That the restriction $\pi^Z$ is smooth is a general fact for subsets
considered with the subset diffeology. Indeed, the subset diffeology
consists of precisely those plots $p:U\to Z$ whose compositions with
the obvious inclusion map is a plot of the ambient space; this
essentially means that this is the subset of $\calD_V$ consisting of
precisely those plots whose range is contained in $Z$. Since the
composition of any plot of $\calD_V$ with $\pi$ is a plot of $X$ by
smoothness of $\pi$, this holds automatically for the restricted
subset of them (considering also that $\pi^Z$ is just the
restriction of $\pi$ to $Z$).

Furthermore, all fibres of $Z$ carry a vector space structure by the
definition given. It remains to check that the two operations are
smooth; and, as above, this follows from the definition of the
subset diffeology and the fact that the two operations are the
restrictions (to $Z\times_X Z$ and $Z$ respectively) of the
corresponding operations on $V$. Finally, the smoothness of the zero
section is automatic since it takes values in $Z$ (its composition
with any given plot of $X$ is by assumption a plot of $V$, which
takes values in $Z$, so it is a plot for the subset diffeology of
$Z$).
\end{proof}

There is also a kind of \emph{vice versa} version of this lemma,
which is actually quite useful, that we now prove.

\begin{lemma}\label{any:subspaces:bundle:lem}
Let $\pi:V\to X$ be a diffeological vector pseudo-bundle, and let
$W\subset V$ be such that $\pi|_W:W\to X$ is surjective, and for
every $x\in X$ the pre-image $W_x=\pi^{-1}(x)\cap W$ is a vector
subspace of $V_x=\pi^{-1}(x)$. Then $W$ is a diffeological vector
pseudo-bundle for the subset diffeology of $V$.
\end{lemma}

We note that the surjectivity of $\pi_W$ follows from the second
condition (a vector subspace is never empty, so neither is $W_x$)
and therefore is superfluous; we leave it in for reasons of
readability.

\begin{proof}
The smoothness of $\pi|_W$ and of the operations follows from the
fact that they are restrictions of smooth maps on, respectively,
$V$, $V\times_X V$, and $\matR\times V$; in the first case it is a
direct consequence of the definition of the subset diffeology, and
in the other two we use the fact that the subset diffeology is
well-behaved with respect to direct products (as follows from the
definition of the direct product diffeology).
\end{proof}

The above lemma illustrates once again the extreme flexibility of
diffeology:\footnote{Coming with its own price, of course.} contrary
to what one would normally expect from a differentiable setting,
\emph{any} collection of subspaces of a vector pseudo-bundle, one
for each fibre, forms naturally a diffeological sub-bundle.

\paragraph{Quotients} Let $\pi:V\to X$ be a diffeological vector
pseudo-bundle, and let $\pi^Z:Z\to X$ be a sub-bundle of it.
Fibrewise, we can define the quotients $\pi^{-1}(x)/(\pi^Z)^{-1}(x)$
for every $x\in X$ (it makes sense to write here for brevity
$V_x/Z_x$), since each space of the sort is the quotient of a
diffeological vector space (of $V_x$) over its subspace, $Z_x$. Now,
both spaces have subset diffeologies, which we denote by $\calD_x^V$
and $\calD_x^Z$ respectively, so their quotient has the
corresponding quotient diffeology, denoted by $\calD_x^{V/Z}$. Let
$\calD_X^{V/Z}$ be the \textbf{finest} diffeology on the set
$W=\cup_{x\in X}V_x/Z_x$ such that the subset diffeology on every
$V_x/Z_x$ contains the diffeology $\calD_x^{V/Z}$.

On the other hand, the same structure of a vector space and a
subspace of it on each fibre can be seen as an equivalence relation
on $V$; the corresponding quotient is again $W$. However, formally
at least, the quotient diffeology $\calD^{V/Z}$ on $W$, relative to
the diffeology of $V$ and this equivalence relation, might be
different from the diffeology $\calD_X^{V/Z}$. We now formally
establish that they are the same.

\begin{lemma}
Let $\pi:V\to X$ be a diffeological vector pseudo-bundle, and let
$\pi^Z:Z\to X$ be a sub-bundle of it. The diffeologies
$\calD_X^{V/Z}$ and $\calD^{V/Z}$ on the quotient pseudo-bundle
$\pi^{V/Z}:W\to X$ coincide.
\end{lemma}

\begin{proof}
Recall that, by definition, $\calD^{V/Z}$ is the finest diffeology
for which $\pi^{V/Z}$ is smooth. In other words, $p':U\to W$ is a
plot for $\calD^{V/Z}$ if and only if locally it is the composition
of some plot $p$ of $\calD_V$ with the quotient projection
$\tilde{\pi}:V\to W$, that is, if $U$ is small enough, we have
$p'=\tilde{\pi}\circ p$. Recall also that by definition of
$\tilde{\pi}$ we have the equality $\pi=\pi^{V/Z}\circ\tilde{\pi}$
everywhere on $V$.

Now, what we need to show to prove the lemma, is that the
restriction of $p'$ on each fibre $(\pi^{V/Z})^{-1}=V_x/Z_x$ of
$\pi^{V/Z}:W\to X$ is a plot for the quotient diffeology (of
diffeological vector spaces) of $V_x/Z_x$, that is, that locally it
is a composition of the projection $V_x\to Z_x$ with a plot of
$V_x$. Now, the projection just-mentioned is the restriction to
$V_x$ of $\tilde{\pi}$, so to establish the claim it suffices to
consider the restriction of $p$ to $V_x$, which is a plot for the
latter by definition of the subset diffeology; restricting the
equality $p'=\tilde{\pi}\circ p$ to the fibre $V_x$, we get the
desired statement.
\end{proof}

\begin{rem}
Putting this lemma together with Lemma
\ref{any:subspaces:bundle:lem} yields a useful (from practical point
of view, at least) conclusion: every quotient of a diffeological
vector pseudo-bundle that preserves the operations is a quotient
pseudo-bundle over a diffeological vector sub-bundle. Indeed, the
condition that operations be preserved simply tells us that on each
fibre the kernel of the quotient is a vector subspace. The
collection of these subspaces which is considered with the subset
diffeology is a sub-bundle by Lemma \ref{any:subspaces:bundle:lem},
and the choice of diffeologies is consistent everywhere by the lemma
established above.
\end{rem}

\paragraph{The direct product of pseudo-bundles} Let $\pi_1:V_1\to X$,
$\pi_2:V_2\to X$ be two diffeological vector space pseudo-bundles
with the same base space. The total space of the product bundle is
of course that of the usual product bundle: $V_1\times_X
V_2=\cup_{x\in X}(V_1)_x\times(V_2)_x$. The \textbf{product bundle
diffeology} (see \cite{vincent}, Definition 4.3.1) is the coarsest
diffeology such that the fibrewise defined projections are smooth;
this diffeology includes, for instance, for each $x\in X$ all maps
of form $(p_1,p_2)$, where $p_i:U\to(V_i)_x$ is a plot of $(V_i)_x$
for $i=1,2$.

The following is a more concrete description of the product bundle,
in the following way. Consider the direct product $V_1\times V_2$
and endow it with the product diffeology. The direct product bundle
$V_1\times_X V_2$ consists, as a set, of all pairs $(x_1,x_2)$ such
that $\pi_1(x_1)=\pi_2(x_2)$. It is endowed with the subset
diffeology relative to that of $V_1\times V_2$; it is precisely the
diffeology that is described above. The map $\pi$ is induced by
$\pi_1$ and $\pi_2$ in the obvious way; its smoothness follows
immediately from the smoothness of these two maps.

\paragraph{The direct sum of pseudo-bundles} It suffices to add to
the above pseudo-bundle $V_1\times_X V_2\to X$ the obvious
operations to get the direct sum of pseudo-bundles, which we denote
by $\pi_{\oplus}:V_1\oplus_X V_2\to X$ (sometimes writing simply
$V_1\oplus V_2\to X$). More precisely, the addition operation is
defined as a map $(V_1\oplus V_2)\times_X(V_1\oplus V_2)\to X$ (with
the obvious addition on fibres,
$(v_1,v_2)+(v_1',v_2')=(v_1+v_1',v_2+v_2')$); the scalar
multiplication map is a map $\matR\times(V_1\oplus V_2)\to X$.

The following statement (also deduced from \cite{vincent}) is
obvious; we cite it for completeness.

\begin{lemma}
The pseudo-bundle $\pi_{\oplus}:V_1\oplus V_2\to X$ is a
diffeological vector pseudo-bundle with respect to the operations
just described.
\end{lemma}

\begin{rem}
There is the following \emph{a priori} question. Let $\pi:V\to X$ be
a diffeological vector pseudo-bundle (with finite-dimensional
fibres), and let $\pi^Z:Z\to X$ be a sub-bundle of it. Then $Z_x$ is
a vector subspace of $V_x$ for every $x$, so we can find a subspace
$W_x\leqslant V_x$ such that $V_x=Z_x\oplus W_x$ (as a vector space,
without considering the diffeology). As shown above (Lemma
\ref{any:subspaces:bundle:lem}), the collection $W=\cup_{x\in X}$
defines another sub-bundle of $V$, and such that pointwise $V$
splits as a direct sum of these two pseudo-bundles. Does it also
split as a diffeological vector pseudo-bundle? Immediately below we
give an example that answers this question in the negative.
\end{rem}

\begin{example}
Here is an example that illustrates that a diffeological vector
pseudo-bundle may split as a vector (pseudo-)bundle in the category
of vector spaces, but not in the category of diffeological vector
spaces. Let $\pi:\matR^4\to\matR$ be the usual projection of
$\matR^4$ onto its first coordinate (thus, the target space of $\pi$
is identified with the first coordinate axis of $\matR^4$). Endow
$\matR^4$, the total space, with the pseudo-bundle diffeology
generated by the plot $p:\matR^2\to\matR^4$ acting by
$p(x,y)=(x,0,x|y|,x|y|)$; let $Z$ be the sub-bundle given, as a
subset of $\matR^4$, by the equation $x_4=0$. Now, observe the
following: for each fixed $x\neq 0$ (if $x=0$ then we just obtain
the standard $\matR^4$ with the standard $\matR^3$ in it, composed
of the first three coordinates) the fibre $\pi^{-1}(x)$ is the
diffeological vector space with $\matR^3$ as the underlying space,
endowed with the diffeology generated (in the vector space
diffeology sense) by the plot $y\mapsto(0,|y|,|y|)$. This is pretty
much the same as an example made in \cite{pseudometric}, where it is
shown that in such a space the subset diffeology of $Z_x$, when
summed with the subset diffeology of its standard orthogonal
complement $W_x$ (which, for each fixed $x$, is described by setting
$x_2=x_3=0$), does not give back the original diffeology of the
ambient space $V_x$. Thus, if we define $W$ to be the sub-bundle
that as a set is given by the equations $x_2=x_3=0$, then as a
pseudo-bundle $\pi:V\to X$ does split into the direct sum
$\pi^Z\oplus\pi^W:Z\oplus W\to X$, but the direct sum diffeology on
the latter is the standard one, while that on the former
pseudo-bundle is obviously not. We conclude that the splitting of
$\pi$ into the direct sum $\pi^Z\oplus\pi^W$ does exist but it is
not diffeologically smooth.
\end{example}

\paragraph{The tensor product pseudo-bundle} This notion was described in
\cite{vincent} (see Definition 5.2.1); we give a more explicit
description than the one that appears therein. Consider the direct
product bundle $\pi:V_1\times_X V_2\to X$ defined in the previous
paragraph. The pre-image $\pi^{-1}(x)=(V_1\times_X V_2)_x$ of any
point $x\in X$ is obviously the vector space $(V_1)_x\times(V_2)_x$;
let $\phi_x:(V_1)_x\times(V_2)_x\to(V_1)_x\otimes(V_2)_x$ be the
universal map onto the corresponding tensor product. The collection
of maps $\phi_x$ defines a map $\phi:V_1\times_X V_2\to V_1\otimes_X
V_2=:\cup_{x\in X}(V_1)_x\otimes_X(V_2)_x$. Let also $Z_x$ be the
kernel of $\phi_x$; by Lemma \ref{any:subspaces:bundle:lem} the
collection $Z$ of subspaces $Z_x$ for all $x\in X$ yields a
(diffeological) sub-bundle of $V$, and there is a well-defined
quotient pseudo-bundle $(V_1\times_X V_2)/Z$, where each fibre is
the (diffeological) tensor product $(V_1)_x\otimes(V_2)_x$.

\begin{defn}
The \textbf{tensor product bundle diffeology} on the tensor product
bundle $V_1\otimes_X V_2$ is the finest diffeology on $V_1\otimes_X
V_2$ that contains the pushforward of the diffeology of $V_1\times_X
V_2$ by the map $\phi$; equivalently, it is the quotient diffeology
on the quotient pseudo-bundle $\pi^{(V_1\times_X
V_2)/Z}:(V_1\times_X V_2)/Z=V_1\otimes V_2\to X$.
\end{defn}

Due to the considerations we have made on sub-bundles and quotient
bundles, this definition is well-posed, in the sense that the two
ways to give it, said to be equivalent, are indeed so.

\paragraph{The dual pseudo-bundle} It remains to define the dual bundle (once
again, this definition is available in \cite{vincent}, Definition
5.3.1). Let $\pi:V\to X$ be a diffeological vector space; fibrewise,
the dual space of it is constructed by taking the union $\cup_{x\in
X}(\pi^{-1}(x))^*=:V^*$ (where $(\pi^{-1}(x))^*$ is the
diffeological dual) with the obvious projection, which we denote
$\pi^*$. What is essential here is to define the diffeology with
which $V^*$ is endowed (which will also describe the topological
structure of the underlying space $V^*$, via the concept of
D-topology underlying the diffeological structure chosen).

\begin{defn}
The \textbf{dual bundle diffeology} on $V^*$ is the finest
diffeology on $V^*$ such that: 1) the composition of any plot with
$\pi^*$ is a plot of $X$ (equivalently, $\pi^*$ is smooth); 2) the
subset diffeology on each fibre coincides with its diffeology as the
diffeological dual $(\pi^{-1}(x))^*$ of fibre $\pi^{-1}(x)$.
\end{defn}

The following curious example illustrates how much things can change
in the diffeological setting (with respect to the usual one).

\begin{example}
Let $\pi:V\to X$ be the diffeological vector pseudo-bundle over the
space $X$ of Example \ref{axes:by:gluing:wire:ex}, that we have
constructed in the Example \ref{coarse:fibres:ex}. Since all fibres
have coarse diffeology, their diffeological duals are always zero
spaces, which means that the dual bundle in this case is just a
trivial covering (in the usual sense) of $X$ by itself.
\end{example}

One matter that should be discussed on the basis of the above
definition is why such diffeology exists. The idea is to start with
the pullback of the diffeology of $X$ by the map $\pi^*$; in this,
we shall take the smallest sub-diffeology
$\calD^*\subset(\pi^*)^*(\calD_X)$ that contains all plots of
individual fibres, that is, maps $U\to(\pi^{-1}(x))^*$ that are
plots for the functional diffeology of the diffeological dual
$(\pi^{-1}(x))^*$, for all $x$, and such that the map $\pi^*$ be
smooth. Applying now the lattice property of the family of all
diffeologies on a given set (this is mentioned in Section 1, see
\cite{iglesiasBook}, Section 1.25 for the detailed treatment), we
conclude that the diffeology $\calD^*$ is well-defined. In addition,
while an \emph{a priori} question could be, whether the subset
diffeology relative to $\calD^*$ on each fibre is indeed its
functional diffeology (which is the standard diffeology of a
diffeological dual, see \cite{vincent}, \cite{wu}),\footnote{As
opposed to being strictly coarser; by construction these are the
only options.} part (3) of the proof of Proposition 5.3.2 in
\cite{vincent} asserts that the answer is positive (and so the
fibres are indeed diffeologial duals).

\paragraph{Explicit description of the dual bundle diffeology} This
description is available in \cite{vincent}, see Definition 5.3.1.
This definition, together with Proposition 5.3.2 (of the same
source), yields the following criterion, a bit cumbersome, but
essential from the practical point of view (see examples that
follow).

\begin{lemma}\label{plots:dual:bundle:descr:lem}
Let $U$ be a domain of some $\matR^l$. A map $p:U\to V^*$ is a plot
for the dual bundle diffeology on $V^*$ if and only if for every
plot $q:U'\to V$ the map $Y'\to\matR$ acting by $(u,u')\mapsto
p(u)(q(u'))$, where $Y'=\{(u,u')|\pi^*(p(u))=\pi(q(u'))\in
X\}\subset U\times U'$, is smooth for the subset diffeology of
$Y'\subset\matR^{l+\dim(U')}$ and the standard diffeology of
$\matR$.
\end{lemma}

As an illustration, let us apply this definition to the following
example.

\begin{example}\label{dual:diffeology:plots:ex}
Take $V=\matR^2$ endowed with the pseudo-bundle diffeology generated
by the plot $q$ acting as $(x,y)\mapsto(x,|xy|)$. Define $X$ to be
the standard $\matR$, and let $\pi$ be the projection of $V$ onto
its first coordinate.\footnote{This is the same pseudo-bundle we
have already seen in Example \ref{top:trivial:dif:nontrivial:ex}.}

Now let us apply Lemma \ref{plots:dual:bundle:descr:lem} to describe
the dual bundle. Let $p:U\to V^*$ be a (prospective) plot of the
dual bundle. It is convenient to write $p(u)$ as an element of the
form $p(u)=(p_1(u),p_2(u))$, where $p_1(u)=\pi^*(p(u))$ determines
the fibre to which $p(u)$ belongs, while $p_2(u)$ determines the
corresponding element of $(V_{p_1(u)})^*$. We note right away that
$\pi^*$ being smooth is thus equivalent to $u\mapsto p_1(u)$ being
the usual smooth function.

If for a plot of $V$ we have a linear combination of constants, this
means that we have essentially a usual smooth map $f:U'\to\matR^2$;
we write $u'\mapsto(f_1(u'),f_2'(u'))$. The subset $Y'$ is composed
of all pairs $(u,u')$ such that $p_1(u)=f_1(u')$; the corresponding
evaluation of $p$ on $f$ is $(u,u')\mapsto p_2(u)\cdot f_2(u')$, and
this should be smooth for the subset diffeology of
$Y'\subset\matR^m$ for a suitable $m$. Since $f$ can be \emph{any}
smooth map (thus, it could be identically a non-zero constant), this
implies that $p_2$ must be a smooth function as well.

Consider now the evaluation of $p$ on the plot $q$ that defines the
diffeology of $V$. We have $p(u)(q(x,y))=p(u)(x,|xy|)=|xy|p_2(u)$,
and this is defined on $Y'=\{(u,(x,y))|p_1(u)=x\}$ and must extend
to an ordinary smooth map defined on $U\times\matR^2$. This implies
that $p_2$ is identically zero outside of the subset $p_1^{-1}(0)$,
and it is any smooth function in the interior of this subset. Thus,
to define a plot $p$ of the dual bundle we can first choose any
smooth function $p_1:U\to\matR$, and then choose another smooth
function $p_2:U\to\matR$ such that $p_2$ has non-zero values only on
the interior of $p_1^{-1}(0)$;\footnote{This interior of course
might be empty, in which case we'll be forced to set $p_2=0$
everywhere; on the other hand, there are plenty of smooth functions
such that the pre-image of zero has non-empty interior.} we obtain
in this way $p(u)=(p_1(u),p_2(u))$ which satisfies all the desired
conditions.\footnote{In what concerns the compositions of $q$ with
smooth and linear combinations of such, the reasoning just made
easily extends to those; we omit the details for reasons of
brevity.}

Finally, do note that that if the interior of $p_1^{-1}(0)$ is
non-empty then we can find a small enough open set $U_1\subset
p_1^{-1}(0)$ and use an appropriate partition of unity to construct
$p_2:U\to\matR$ that satisfies all the properties required and whose
restriction to $U_1$ is any chosen smooth function. This essentially
implies what we wanted, namely, that the diffeology on the fibre at
$0$ be the fine diffeology of $\matR$, while elsewhere it is the
obvious coarse diffeology of the one-point space.
\end{example}

\subsection{Do the operations commute with gluing?}

A natural question that presents itself at this point is whether the
operations described in the previous section commute with gluing;
or, more precisely, when they do so.

\paragraph{Sub-bundles} The easiest case is that of a
sub-bundle (it also illustrates why the question is not entirely
trivial, since it is rather obvious that the gluing does not
necessarily preserve sub-bundles. Here is the precise situation that
we wish to consider.

Let $\pi_1:V_1\to X_1$ and $\pi_2:V_2\to X_2$ be two diffeological
vector pseudo-bundles, and let $\pi_1^Z:Z_1\to X_1$ and
$\pi_2^Z:Z_2\to X_2$ be their respective vector
sub-bundles.\footnote{``Sub-pseudo-bundles'' would be more precise,
but we avoid this term, as we have already done for other terms
similar, as it sounds unnecessarily complicated (the object that we
are considering is clear from the context).} Recall that this means
that each of $Z_1$, $Z_2$ is a subspace of respectively $V_1$, $V_2$
such that: 1) $Z_i$ has non-empty intersection with each fibre of
$\pi_i$ (in other words, the restriction $\pi_i^Z$ of $\pi_i$ to
$Z_i$ is onto $X_i$), 2) moreover, for $i=1,2$ and for every $x\in
X_i$ the intersection $Z_i\cap\pi_i^{-1}(x)$ is a vector subspace of
$\pi_i^{-1}(x)$, and finally, 3) the diffeology of $Z_i$ is the
subset diffeology relative to the diffeology of $V_i$.

In the situation just described it is quite easy to see that in
order to even ask the question we must impose the following
condition: $\tilde{f}(\pi_1^{-1}(x)\cap
Z_1)\subset(\pi_2^{-1}(f(x))\cap Z_2$ for every $x\in Y$. This
condition just states that for every fibre of $Z_1$ in the domain of
the definition of $\tilde{f}$ this fibre should be sent to a
subspace of the corresponding fibre of $Z_2$; in other words,
$\tilde{f}$ must restrict to a gluing between $Z_1$ and $Z_2$. It is
quite obvious that this is not automatic, but here is an example
where it does not happen.

\begin{example}
Let $\pi_1:\matR^3\to\matR$ be the projection to the first
coordinate, with both spaces endowed with the standard diffeologies,
and let $\pi_2:\matR^3\to\matR$ be another copy of the same
pseudo-bundle (a true bundle, really). Let $Y=\{0\}$, and let $f$
and $\tilde{f}$ be the obvious identity maps. Let $Z_1$ be the plane
given by the equation $z=0$ (the $(x,y)$-coordinate plane); it
suffices to take $Z_2$ to be the $(x,z)$-coordinate plane of the
second copy of $\matR^3$ to get an example where the image, in the
glued pseudo-bundle $\matR^3\cup_{\tilde{f}}\matR^3$, of $Z_1\sqcup
Z_2$ is not a vector pseudo-bundle at all (indeed, the fibre at $0$
is the union of two lines with a one-point intersection,
\emph{i.e.}, not a vector space).
\end{example}

\paragraph{Quotients} The situation of quotients is similar to that
of sub-bundles: the question that we should really ask is, under
which conditions a given gluing $(\tilde{f},f)$ of two
pseudo-bundles yields a well-defined gluing on their given
quotients? This question is a standard, and the answer to it is also
standard: this happens if and only if $(\tilde{f},f)$ induces a
well-defined gluing over the sub-bundles that are kernels of the
quotients. The condition for this has been stated just above, and we
do not repeat it.

\paragraph{Direct product/sum} Let us now consider the behavior of
the gluing with respect to the direct product. Suppose we have the
following two pairs of vector pseudo-bundles, one composed of
$\pi_1:V_1\to X_1$ and $\pi_2:V_2\to X_2$ glued together along some
appropriate choice of $f:X_1\supset Y\to X_2$ and
$\tilde{f}:\pi_1^{-1}(Y)\to\pi_2^{-1}(f(Y))$, to form the
pseudo-bundle $\pi:V_1\cup_{\tilde{f}}V_2\to X_1\cup_f X_2$. The
other pair is $\pi_1':V_1'\to X_1$ and $\pi_2':V_2'\to X_2$, with
gluing on the bases along the same map $f:X_1\supset Y\to X_2$ and
the lift $\tilde{f}':(\pi_1')^{-1}(Y)\to(\pi_2')^{-1}(f(Y))$; this
gives the pseudo-bundle $\pi':V_1'\cup_{\tilde{f}'}V_2'\to X_1\cup_f
X_2$. Let us now take also the direct product bundles
$\pi_{1,\times}:V_1\times_{X_1} V_1'\to X_1$ and
$\pi_{2,\times}:V_2\times_{X_2} V_2'\to X_2$, and consider the
question whether the two given gluings, $(\tilde{f},f)$ and
$(\tilde{f}',f)$ induce naturally a gluing between these two
products.

The gluing of the bases is simply inherited (since the bases are in
fact the same); it is given by the same $f$. Now, if an appropriate
lift of it (which we denote by $\tilde{f}_{\times}$) exists, it must
be, as we know, defined on the whole of $\pi_{1,\times}^{-1}(Y)$, so
let us first say what it is. It is quite easy to see that this is
$\pi_1^{-1}(Y)\times_Y\pi_2^{-1}(Y)$ (since each of the two sets is
composed of the whole fibres). Analogously,
$\pi_{2,\times}^{-1}(f(Y))$ decomposes as the (fibrewise) direct
product $(\pi_1')^{-1}(f(Y))\times_{f(Y)}(\pi_2')^{-1}(f(Y))$. It
follows then that $\tilde{f}_{\times}$ can be defined as the
(fibrewise) product of the maps $\tilde{f}$ and $\tilde{f}'$,
namely, for any $v_1\in\pi_1^{-1}(Y)$ and $v_1'\in(\pi_1')^{-1}(Y)$
we set
$\tilde{f}_{times}(v_1,v_1')=(\tilde{f}(v_1),\tilde{f}'(v_1'))$.
This is well-defined and satisfies all the desired conditions, by
definition of the product diffeology (both in the case of
pseudo-bundles and in the case of individual vector spaces).
Furthermore, if we add the operations so as to obtain the direct
sum, these are obviously going to be smooth. What all this means can
be summarized as follows:

\begin{lemma}\label{gluing:commutes:product:lem}
Let $\pi:V_1\cup_{\tilde{f}}V_2\to X_1\cup_f X_2$ be a diffeological
vector pseudo-bundle obtained by gluing together pseudo-bundles
$\pi_1:V_1\to X_1$ and $\pi_2:V_2\to X_2$ along a smooth map
$f:X_1\supset Y\to X_2$ and its smooth linear lift
$\tilde{f}:\pi_1^{-1}(Y)\to\pi_2^{-1}(f(Y))$, and let
$\pi':V_1'\cup_{\tilde{f}'}V_2'\to X_1\cup_f X_2$ be another
diffeological vector pseudo-bundle obtained by gluing together
pseudo-bundles (over the same respective bases) $\pi_1':V_1'\to X_1$
and $\pi_2':V_2'\to X_2$ along the same smooth map $f:Y\to X_2$ and
its smooth linear lift
$\tilde{f}':(\pi_1')^{-1}(Y)\to(\pi_2')^{-1}(f(Y))$. Then the
following two pseudo-bundles are diffeomorphic as pseudo-bundles:
$$\pi_{\times}:(V_1\cup_{\tilde{f}}V_2)\times_{X_1\cup_f X_2}(V_1'\cup_{\tilde{f}'}V_2')\to X_1\cup_f X_2, \mbox{ and} $$
$$\pi_{1,\times}\cup_{(\tilde{f}_{\times},f)}\pi_{2,\times}:
(V_1\times_{X_1}V_1')\cup_{\tilde{f}_{\times}}(V_2\times_{X_2}V_2')\to
X_1\cup_f X_2.$$
\end{lemma}

The proof is obvious from the construction.

\paragraph{Tensor product} Let us now turn to the tensor product; by
its definition, it suffices to apply Lemma
\ref{gluing:commutes:product:lem} and observe that the nuclei are
preserved by $\tilde{f}$, $\tilde{f}'$, and $\tilde{f}_{\times}$. So
we get an almost complete analogue of Lemma
\ref{gluing:commutes:product:lem}, namely, the following statement:

\begin{lemma}\label{gluing:commutes:tensor:lem}
Let $\pi:V_1\cup_{\tilde{f}}V_2\to X_1\cup_f X_2$ be a diffeological
vector pseudo-bundle obtained by gluing together pseudo-bundles
$\pi_1:V_1\to X_1$ and $\pi_2:V_2\to X_2$ along a smooth map
$f:X_1\supset Y\to X_2$ and its smooth linear lift
$\tilde{f}:\pi_1^{-1}(Y)\to\pi_2^{-1}(f(Y))$, and let
$\pi':V_1'\cup_{\tilde{f}'}V_2'\to X_1\cup_f X_2$ be another
diffeological vector pseudo-bundle obtained by gluing together
pseudo-bundles (over the same respective bases) $\pi_1':V_1'\to X_1$
and $\pi_2':V_2'\to X_2$ along the same smooth map $f:Y\to X_2$ and
its smooth linear lift
$\tilde{f}':(\pi_1')^{-1}(Y)\to(\pi_2')^{-1}(f(Y))$. Then the
following two pseudo-bundles are diffeomorphic as pseudo-bundles:
$$\pi_{\otimes}:(V_1\cup_{\tilde{f}}V_2)\otimes_{X_1\cup_f X_2}(V_1'\cup_{\tilde{f}'}V_2')\to X_1\cup_f X_2, \mbox{ and} $$
$$\pi_{1,\otimes}\cup_{(\tilde{f}_{\otimes},f)}\pi_{2,\otimes}:
(V_1\otimes_{X_1}V_1')\cup_{\tilde{f}_{\otimes}}(V_2\otimes_{X_2}V_2')\to
X_1\cup_f X_2,$$ where the map $\tilde{f}_{\otimes}$ is induced by
$\tilde{f}$ at the passage to the tensor product.
\end{lemma}

\paragraph{Duals} Let us now turn to the question of dual
pseudo-bundles. Once again, assume that we have two vector
pseudo-bundles, $\pi_1:V_1\to X_1$ and $\pi_2:V_2\to X_2$, which are
glued together along a smooth map $f:X_1\supset Y\to X_2$ and its
smooth (fibrewise-)linear lift
$\tilde{f}:\pi_1^{-1}(Y)\to\pi_2^{-1}(f(Y))$, yielding the
pseudo-bundle $\pi:V_1\cup_{\tilde{f}}V_2\to X_1\cup_f X_2$. We wish
to discuss the question of whether (and if so, how) this induces a
gluing between the dual bundles.

Let us start by this observation. Consider an arbitrary point $y\in
Y$; the restriction of $\tilde{f}$ to its pre-image is a smooth
linear map $\pi_1^{-1}(y)\to\pi_2^{-1}(f(y))$ between two
diffeological vector spaces. Then (see \cite{wu}; some details can
also be found in \cite{multilinear}) there is a natural (smooth and
linear) dual map
$\tilde{f}^*:(\pi_2^{-1}(f(y)))^*\to(\pi_1^{-1}(y))^*$, which is
defined in the usual way (that is, by the rule
$\tilde{f}^*(v_2^*)(w_1)=v_2^*(\tilde{f}(w_1))$). However, since it
goes in the opposite (with respect to $\tilde{f}$) direction, it
obviously cannot be a lift of the existing map $f:Y\to X_2$.

The most natural, and the most obvious, way to resolve the situation
is to restrict at this point our discussion to gluings along
\emph{injective} $f$'s. Indeed, assuming that $f$ is injective, we
easily observe that $\tilde{f}^*$ is a lift of its inverse $f^{-1}$.

However, there is still a further condition to impose. Namely,
consider some $y\in Y$; write for brevity $W_1$ to denote
$\pi_1^{-1}(y)$ and $W_2$ to denote $\pi_2^{-1}(\tilde{f}(y))$. Then
it is easy to see that in the pseudo-bundle
$\pi^*:(V_1\cup_{\tilde{f}}V_2)^*\to X_1\cup_f X_2$ the fibre over
$y=f(y)$ is $(W_2)^*$, while in the pseudo-bundle
$\pi_2^*\cup_{(\tilde{f}^*,f^{-1})}\pi_1^*:V_2^*\cup_{\tilde{f}^*}V_1^*\to
X_2\cup_{f^{-1}}X_1$ the fibre over the same point is $(W_1)^*$. It
follows that one necessary condition for these two pseudo-bundles to
be (fibrewise) diffeomorphic is that for every $y\in Y$ the fibres
$\pi_1^{-1}(y)$ and $\pi_2^{-1}(f(y))$ have diffeomorphic
duals.\footnote{Recall that this does not imply that the spaces
themselves should be diffeomorphic; they may easily not be so, such
as in the case of the standard $\matR^2$ and $\matR^3$ with the
diffeology of Example \ref{Rn:vector:space:nontrivial:ex}.}

Thus, the final statement we arrive to is as follows.

\begin{lemma}\label{when:gluing:dual:commute:lem}
Let $\pi_1:V_1\to X_1$ and $\pi_2:V_2\to X_2$ be two diffeological
vector pseudo-bundles, let $f:Y\to X_2$ be a smooth injective map
defined on a subset $Y\subset X_1$, and let
$\tilde{f}:\pi_1^{-1}(Y)\to\pi_2^{-1}(f(Y))$ be its smooth fibrewise
linear lift. Suppose that the restrictions of the corresponding dual
bundles $\pi_1^*$ and $\pi_2^*$ to $Y$ and $f(Y)$ respectively are
diffeomorphic. Then the dual map
$\tilde{f}^*:(\pi_2^{-1}(f(Y)))^*\to(\pi_1^{-1}(Y))^*$ is a smooth
fibrewise linear lift of the map $f^{-1}:f(Y)\to Y$ to the dual
pseudo-bundles $\pi_2^*:V_2^*\to X_2$ and $\pi_1^*:V_1^*\to X_1$,
and the following two pseudo-bundles are diffeomorphic as
pseudo-bundles:
$$\pi^*:(V_1\cup_{\tilde{f}}V_2)^*\to X_1\cup_f X_2, \mbox{ and}$$
$$\pi_2^*\cup_{(\tilde{f}^*,f^{-1})}\pi_1^*:V_2^*\cup_{\tilde{f}^*}V_1^*\to X_2\cup_{f^{-1}}X_1.$$
\end{lemma}

\begin{proof}
We have already recalled (the known fact) that $\tilde{f}^*$ is
smooth and linear on each fibre; let us formally check that it is
indeed a lift of $f^{-1}$. This means that we need to check the
equality $f^{-1}\circ\pi_2^*=\pi_1^*\circ\tilde{f}^*$ on the
relevant domains of the definition. So let $y\in Y$, and let
$v_2^*\in(\pi_2^*)^{-1}(f(y))$; by definition, $\tilde{f}^*(v_2^*)$
can be viewed as the composition $v_2^*\circ\tilde{f}$. Thus,
evaluated at $v_2^*$, the left-hand side of the equality that we
need to check yields $y$, while the right-hand side becomes
$\pi_1^*(v_2^*\circ\tilde{f})$ and thus also gives $y$ by
injectivity of $f$. The two maps $\tilde{f}^*$ and $f^{-1}$
therefore satisfy the conditions necessary so that gluing can be
done along them; so we turn to the second part of the statement.

Let us now construct the desired diffeomorphism between the
pseudo-bundles. Let us denote for brevity the second map by
$\pi_{gl}^*$. The diffeomorphism $\varphi$ between the bases is
obvious (it is in fact standard); it is given by identity for $x\in
X_1\setminus Y$ and $x\in X_2\setminus f(Y)$, while for $y=f(y)\in
Y/\sim$ its image is the point that formally writes as $y'=f^{-1}\in
f(Y)/\sim$ for $y'=f(y)$. Let us construct now its covering
$\tilde{\varphi}:(V_1\cup_{\tilde{f}}V_2)^*\to
V_2^*\cup_{\tilde{f}^*}V_1^*$.

The idea behind the construction is immediately clear, of course.
The already-existing diffeomorphism between bases gives a one-to-one
correspondence between (whole) fibres of the spaces
$(V_1\cup_{\tilde{f}}V_2)^*$ and $V_2^*\cup_{\tilde{f}^*}V_1^*$.
Now, each fibre of the first space is the dual of some fibre
$\pi^{-1}(x)$ with $x\in X_1\cup_{f}X_2$; as we already noted, if,
say, $x\in X_1\setminus Y$ then $\pi^{-1}(x)$ is a fibre of
$\pi_1:V_1\to X_1$, and its dual is therefore the corresponding
fibre of $\pi_1^*:V_1^*\to X_1$. Furthermore, the image
$\varphi(x)\in X_2\cup_{f^{-1}}X_1$ is essentially $x$ itself, and
its pre-image $(\pi_{gl}^*)^{-1}(\varphi(x))$ is actually the
corresponding fibre of $V_1^*\to X_1$. The same reasoning obviously
applies to any point of $X_2\setminus f(Y)$; so the construction of
$\tilde{\varphi}$ should only be checked for $Y/\sim\subset
X_1\cup_f X_2$.

More precisely, let $y\in Y$; then by construction
$(\pi^*)^{-1}(y)=\left((\pi_1^{-1}(y)\sqcup\pi_2^{-1}(f(y)))/_{v_1=\tilde{f}(v_1)}
\right)^*$, whereas
$(\pi_{gl}^*)^{-1}(y)=\left((\pi_2^*)^{-1}(f(y))\sqcup(\pi_1^*)^{-1}(y)\right)/_{v_2^*=\tilde{f}^*(v_2^*)}$.
A diffeomorphism between the two is now a consequence of the
assumptions of the lemma.
\end{proof}

\subsection{Pseudo-metrics}

As we have already recalled elsewhere, a finite-dimensional
diffeological vector space in general does not admit a smooth scalar
product, unless it is a standard space. This obviously implies that
there is not a straightforward counterpart of the notion of a
Riemannian metric on a diffeological space. On the other hand, in
the case of a single vector space there is the ``best possible''
substitute for the notion of a scalar product (which we've called a
pseudo-metric), a ``least degenerate'' smooth symmetric bilinear
form, and this can be extended to a corresponding notion of a
pseudo-metric on a diffeological vector pseudo-bundle.

\paragraph{Definition of a pseudo-metric} The formal definition of
a pseudo-metric on a single diffeological vector space is as
follows.

\begin{defn}
Let $V$ be a diffeological vector space of finite dimension $n$, and
let $\varphi:V\times V\to\matR$ be a smooth symmetric bilinear form
on it. We say that $\varphi$ is a \textbf{pseudo-metric} if the
multiplicity of its eigenvalue $0$ is equal to $n-\dim(V^*)$.
\end{defn}

It is not \emph{a priori} clear, although it is easy to see (see
\cite{pseudometric}), why this definition makes sense, that is, why
such a pseudo-metric always exists, and why it is the best
substitute for the smooth scalar product. It is proven, however, in
\cite{pseudometric}, that for any smooth symmetric bilinear form on
$V$ the multiplicity of its eigenvalue $0$ is \emph{at least}
$n-\dim(V^*)$. Furthermore, we can always find a smooth symmetric
bilinear form such that the multiplicity of $0$ be precisely
$n-\dim(V^*)$.

Formulating the corresponding notion for diffeological vector
pseudo-bundles is then trivial. Stated formally, it is as follows.

\begin{defn}
Let $\pi:V\to X$ be a diffeological vector pseudo-bundle. A
\textbf{pseudo-metric} on $V$ is any smooth section $g$ of the
diffeological vector pseudo-bundle $\pi_{\otimes}^*:V^*\otimes
V^*\to X$ such that for every $x\in X$ $g(x)$ is a pseudo-metric on
$\pi^{-1}(x)$.
\end{defn}

\paragraph{Representing pseudo-metrics} In the examples that we
provide in the rest of this section, we will need to choose a way to
write down pseudo-metrics. What we do is opt for an \emph{ad hoc}
solution (it is not meant to be generally applicable): since our
pseudo-bundles are, from the set-map point of view, maps of form
$\matR^n\to\matR^k$ given by the projection onto the first $k$
coordinates, and so the fibres are of form $\{x\}\times\matR^{n-k}$
and have a vector space structure with respect to the coordinates
$x_{k+1},\ldots,x_n$, an element of the dual bundle can be naturally
written in the form
$((x_1,\ldots,x_k),a_{k+1}e^{k+1}+\ldots+a_ne^n)$, where
$(x_1,\ldots,x_k)$ is the corresponding point of the base space,
$e^{k+1},\ldots,e^n$ are elements of the basis dual to the canonical
one (both in the sense of the standard $\matR^n$), and so the
element $a_{k+1}e^{k+1}+\ldots+a_ne^n$ has an obvious meaning as an
element of the dual space of
$\{(x_1,\ldots,x_k)\}\times\matR^{n-k}$. The possible pseudo-metrics
then, being bilinear forms and so elements of the tensor product of
the dual with itself, can be written, similarly, as
$((x_1,\ldots,x_k),\sum_{i,j=k+1}^na_{ij}e^i\otimes e^j)$ (in both
cases, generally speaking, there will be restrictions on
coefficients to ensure the smoothness and other conditions).

\paragraph{Examples of pseudo-metrics on topologically trivial
pseudo-bundles} In this paragraph we provide two examples. The first
one is chosen so as to be one of the easiest, but not entirely
trivial.

\begin{example}\label{R3:to:R:z-nontrivial:ex}
Let $\pi:\matR^3\to\matR$ be given by $\pi(x,y,z)=x$; let us endow
$\matR$ with the standard diffeology and $\matR^3$ with the
pseudo-bundle diffeology generated by the map
$p:U=\matR^2\to\matR^3$ acting by $p(u_1,u_2)=(u_1,0,|u_2|)$.
Defining $g(x)=(x,(x^2+1)e^2\otimes e^2)$ gives a pseudo-metric on
this pseudo-bundle (where the meaning of the expression is precisely
the one explained in the preceding paragraph). In fact, it is easy
to see that any pseudo-metric on $V$ has, in the same notation, the
form $g(x)=(x,f(x)e^2\otimes e^2)$, where $f:\matR\to\matR$ is a
smooth everywhere positive function.

Let us formally show that $g(x)=(x,(x^2+1)e^2\otimes e^2)$ defines a
smooth section of $V^*\otimes V^*$. By extension of Lemma
\ref{plots:dual:bundle:descr:lem}, we need to evaluate it a plot of
$V\otimes V$ and show that the resulting ($\matR$-valued) function
is smooth for the (subset) diffeology of its domain of definition
and the standard diffeology of $\matR$. What this essentially means
(as follows from the definition of the pseudo-bundle diffeology
generated by a given plot) and symmetricity of each value of $g(x)$,
it suffices to consider the pair $p\otimes c_v$, why by $c_v$ we
means a constant map and $p$ is the generating plot. This
evaluation, which formally writes as
$g(p(u_1,u_2))=(u_1^2+1)e^2(0)e^2(0)$, is obviously the zero map, so
it is smooth.

Finally, we comment why $g$ has the maximal possible rank
everywhere. This is simply because for any $x\in\matR$ the fibre
$\pi^{-1}(x)$ at $x$ is, as a diffeological vector space, just
$\matR^2$ endowed with the (non-standard) vector space diffeology
generated by the map $u'\mapsto(0,|u'|)$. This is a specific case of
a diffeological vector space seen in Example
\ref{Rn:vector:space:nontrivial:ex} (corresponding to $n=2$). As
mentioned in the Example, its diffeological dual has (in this
specific case) dimension $1$; therefore, any smooth bilinear form,
being an element of the tensor product of this dual with itself, has
rank at most $1$. It remains to note that this is precisely the rank
of $g$ at any given point $x\in X$.
\end{example}

The pseudo-bundle in the previous example is a non-standard one, but
it is still a trivial bundle. Let us now consider an instance of a
non locally trivial pseudo-bundle, that of Example
\ref{top:trivial:dif:nontrivial:ex}.

\begin{example}
Recall that we have $\pi:\matR^2\to\matR$, where $\pi$ is the
projection on the first coordinate, $\matR=X$ is standard, and the
diffeology of $\matR^2=V$ is the pseudo-bundle diffeology generated
by the plot $(x,y)\mapsto(x,|xy|)$. Thus, the fibre at zero has
standard diffeology of $\matR$, while elsewhere it has the vector
space diffeology generated by plot $y\mapsto|y|\cdot\mbox{const}$
(thus, it is again a specific case of a non-standard diffeological
vector space described in the Example
\ref{Rn:vector:space:nontrivial:ex}).

Now, since all fibres are $1$-dimensional, a pseudo-metric is
essentially a real-valued function $f$ on $X$, measuring the value
of the corresponding quadratic form on a chosen basis vector, which
in our case we can take, for each fixed $x$, to be $(x,1)$. Let us
first check that the assignment $x\mapsto(x,1)$ defines a smooth
section $s$ of the pseudo-bundle $V\to X$. To do so, we need to take
an arbitrary plot of $X$, \emph{i.e.}, a smooth (in the ordinary
sense) function $h:U\to\matR$ and to check that its composition with
$s$ is a plot of $V$. This is in fact obvious, because $(s\circ
h)(u)=(h(u),1)$, which is a usual smooth function, and we defined
the diffeology of $V$ to be, in particular, a vector space
diffeology, which implies that it includes all smooth functions in
the usual sense. This section being smooth, we can write any
pseudo-metric in the form $x\mapsto(x,f(x)e^2\otimes e^2)$.

Now, formally a pseudo-metric is a smooth section of the
pseudo-bundle $V^*\otimes V^*\to X$; in our case, from the
description of the diffeologies we see (as also indicated in the
Example \ref{Rn:vector:space:nontrivial:ex}) that for $x\neq 0$ the
fibre at $x$ has a trivial dual, while for $x=0$ it is the standard
$\matR$. This implies that the above-considered function $f$ is a
version of the so-called $\delta$-function: $\delta(0)=1$ (or any
other positive constant; we set it equal to $1$ for technical
reasons) and $\delta(x)=0$ for $x\neq 0$.

Since $\delta^2=\delta$, and by definition of the tensor product
diffeology, it is sufficient to check that the assignment
$x\mapsto(x,\delta(x)e^2)$ defines a smooth section of $V^*\to X$,
in other words, that its composition with any plot of $X$ yields a
plot of $V^*$. Now, a plot of $X$ is an ordinary smooth function
$p_1:U\to\matR$, so the composition we must consider is the map
$p:U\to V^*$ acting by $u\mapsto(p_1(u),\delta(p_1(u))e^2)$.

Let us check that $p$ is a plot of $V^*$. We have already
characterized these plots in Example \ref{dual:diffeology:plots:ex};
the condition that we need to check is that the product of the two
functions, that is, $p_1(u)\delta(p_1(u))$, is identically zero on
the whole of $U$. This follows immediately from the definition of
$\delta$, so we can conclude that setting
$g(x)=(x,\delta(x)e^2\otimes e^2)$ defines a pseudo-metric on our
pseudo-bundle $\pi:V\to X$. As an additional observation, we note
that, as follows from the above discussion, \emph{every}
pseudo-metric on this pseudo-bundle is of this form, up to a choice
of positive constant $\delta(0)$.
\end{example}

\paragraph{Pseudo-metrics and gluing} It is quite clear from our
above discussion that, since the gluing is well-behaved with respect
to bundle maps and the operations on vector pseudo-bundles as soon
as appropriate additional conditions are satisfied, it would be so
also with respect to pseudo-metrics. Indeed, given a gluing of two
pseudo-bundles $\pi_1:V_1\to X_1$ and $\pi_2:V_2\to X_2$, each
endowed with a pseudo-metric $g_1$ or, respectively, $g_2$, via the
maps $\tilde{f}$ and $f$, it is sufficient (and necessary, of
course) to have
$$g_1(x_1)(v_1,v_1')=g_2(f(x_1))(\tilde{f}(v_1),\tilde{f}(v_1'))$$
for all $x_1\in X_1$ and $v_1,v_1'\in\pi_1^{-1}(x_1)$.\footnote{This
could be referred to as a $(\tilde{f},f)$-equivariant choice; we
will occasionally say that $g_1$ and $g_2$ are \emph{compatible}
pseudo-metrics.} If this equality is satisfied, there is an obvious
corresponding pseudo-metric on the pseudo-bundle
$V_1\cup_{\tilde{f}}V_2\to X_1\cup_f X_2$ obtained by the gluing.
Here is a simple example of such.

\begin{example}
Let us take for $\pi_1:V_1\to X_1$ and $\pi_2:V_2\to X_2$ two true
bundles, given by taking $V_1=V_2=\matR^3$ with the standard
diffeology, $X_1=X_2=\matR$ again with the standard diffeology, and
finally setting $\pi_1$ to be the projection of $V_1$ onto its first
coordinate, while $\pi_2$ is the projection of $V_2$ onto its second
coordinate. Accordingly, $X_1$ is identified with the $x$-axis of
$V_1$, and $X_2$ is identified with the $y$-axis of $V_2$.

The simplest choice of gluing on the bases is to identify the two
copies of $\matR$ at the origin, so the map $f$ is just the
point-to-point map $f:\{0\}\to\{0\}$. The pre-image $\pi_1^{-1}(0)$,
so the domain of definition of $\tilde{f}$, is the
$(y,z)$-coordinate plane of $V_1$, the set $\{(0,y,z)\}$; the
pre-image $\pi_2^{-1}(0)$, which contains the range of $\tilde{f}$,
is the $(x,z)$-coordinate plane of $V_2$, the set $\{(x,0,z)\}$. We
make the simplest choice possible for $\tilde{f}$, defining it
$\tilde{f}(0,y,z)=(y,0,z)$; this is obviously linear and smooth.

Finally, we choose two compatible pseudo-metrics $g_1$ and
$g_2$.\footnote{Since the diffeology is standard, they are in fact
true metrics.} In the form that we have already explained, we can
write $g_1(x)=(x,e^2\otimes e^2+e^3\otimes e^3)$ (note that there
are no restrictions on the choice of the first pseudo-metric). Then
taking, for instance, $g_2(y)=(y,e^1\otimes e^1+e^3\otimes e^3)$
satisfies the condition of compatibility.
\end{example}

For completeness, we now add a somewhat less trivial, but still
rather simple, example.

\begin{example}
Let $\pi_1:V_1\to X_1$ be a projection of $V_1=\matR^3$ onto its
first coordinate, so $X_1$, which is again a standard $\matR$, is
identified with the $x$-axis of $V_1$. Let $V_1$ be endowed with the
pseudo-bundle diffeology generated by the plot
$(x,y,z)\to(x,y,|z|)$. This immediately implies that this bundle is
diffeologically trivial with non-standard fibre; this fibre is
$\matR^2$ whose (subset) diffeology is generated by the plot
$z\to|z|$, and from Example \ref{Rn:vector:space:nontrivial:ex} we
deduce that its dual is the standard $\matR$.

As for $\pi_2:V_2\to X_2$, we simply take $V_2$ to be the standard
$\matR^2$ and $\pi_2$ the projection onto its $y$-coordinate. Note
that the assumptions of Lemma \ref{when:gluing:dual:commute:lem} are
satisfied. Since again we choose to identify $X_1$ with $X_2$ at
their respective origins, $\tilde{f}$ acts between $\{(0,y,z)\}$ and
$\{(x,0)\}$. Note that it is essentially a smooth (in the sense of
the chosen diffeology on $\matR^2$) linear map $\matR^2\to\matR$,
where the $\matR^2$ is endowed with a non-standard diffeology, while
$\matR$ is just standard. As has already been mentioned, the only
smooth linear maps in this case are smooth multiples of $e^2$, so we
set $\tilde{f}(0,y,z)=(y,0)$. Applying the reasoning already made,
we choose $g_1(x)=(x,e^2\otimes e^2)$ and $g_2(y)=(y,e^1\otimes
e^1)$.
\end{example}

\paragraph{Existence of pseudo-metrics} In this concluding paragraph
we turn to the following absolutely natural question: does a
pseudo-metric always exist? For the definitions given as of now, we
provide an example that shows that the answer is negative.

\begin{example}
Take $V=\matR^4$ endowed with the pseudo-bundle diffeology generated
by the plot $p:\matR^3\to\matR^4$ that acts by the rule
$(x,y,z)\mapsto(x,y,0,y|z|)$, and take $X$ to be the standard
$\matR^2$. Let $\pi$ be the projection of $V$ onto its first two
coordinates; this is obviously smooth, so $\pi:V\to X$ is a
diffeological vector pseudo-bundle.

Now, consider an arbitrary point $(x_0,y_0)\in X$. It is quite
obvious that if $y_0=0$ then the fibre $\pi^{-1}(x_0,y_0)$ at this
point has standard diffeology of $\matR^2$, while if $y_0\neq 0$
then it has the vector space diffeology generated by the plot
$z\mapsto(0,y_0|z|)$. The latter is a particular case (for $n=2$) of
the family of spaces described in Example
\ref{Rn:vector:space:nontrivial:ex}. In particular, we can conclude
that for $y_0=0$ the dual of the corresponding fibre is
$2$-dimensional, while for $y_0\neq 0$ it is $1$-dimensional.

Let us consider a prospective pseudo-metric $g(x,y)$ on this
pseudo-bundle. In our notation $g(x,y)$ is an element of form
$(x,y,a(x,y)e^3\otimes e^3+b(x,y)e^3\otimes e^4+b(x,y)e^4\otimes
e^3+c(x,y)e^4\otimes e^4)$, where the symmetricity has already been
taken into account. By extension of Lemma
\ref{plots:dual:bundle:descr:lem}, we should consider its evaluation
on a generic section of $V\otimes V$.

Observe first of all that the maps $p_1,p_2:\matR^2\to V$ given,
respectively, by $p_1(x,y)=(x,y,1,0)$ and by $p_2(x,y)=(x,y,0,1)$,
are necessarily smooth, to account for constant maps in the subset
diffeology of each fibre in the pseudo-bundle diffeology of $V$.
Evaluating our prospective $g$ on $p_1\otimes p_1$, $p_1\otimes
p_2$, and $p_2\otimes p_2$ allows us to conclude that the functions
$a,b,c$ must be smooth functions in the ordinary
sense.\footnote{Albeit informally, this was something to expect in
view of the fact that the diffeological dual of a finite-dimensional
diffeological vector space is always a standard space.}

Furthermore, evaluating $g$ on $p\otimes p_1$, we obtain
$b(x,y)|z|$, while evaluating it on $p\otimes p_2$, we obtain
$c(x,y)|z|$; both of them must be smooth in the usual sense, that
is, $b$ and $c$ are identically zero functions. This implies that
$g$ never has rank $2$ and therefore does not always give a
pseudo-metric on fibres; more precisely, it does not on the fibres
of form $\pi^{-1}(x,0)$.\footnote{The closest we arrive to is a
smooth rank-$1$ symmetric form, by choosing $a(x,y)$ to be any
smooth everywhere positive function.}
\end{example}

\vspace{1cm}

\noindent University of Pisa \\
Department of Mathematics \\
Via F. Buonarroti 1C\\
56127 PISA -- Italy\\
\ \\
ekaterina.pervova@unipi.it\\

\end{document}